\documentclass[11pt]{article}
\usepackage{amsmath, amsthm, amssymb}
\usepackage{float,epsfig}
\usepackage{subcaption, multirow}
\usepackage{tkz-graph}
\usepackage{amsthm}
\usepackage{color,graphicx, float,lscape, enumerate}
\usepackage[colorlinks = true]{hyperref}
\usepackage{cleveref}
\usepackage{cite}
\usepackage{blkarray}
\usepackage{xcolor}
\usepackage{silence}
\newtheorem{theorem}{Theorem}[section]
\newtheorem{corollary}{Corollary}[theorem]
\newtheorem{lemma}[theorem]{Lemma}
\newtheorem{proposition}[theorem]{Proposition}
\newtheorem{notation}{Notation}[section]
\newtheorem{remark}{Remark}[section]

\newtheorem{definition}{Definition}[section]

\newcommand{\ba}{\begin{array}}
\newcommand{\ea}{\end{array}}

\newcommand{\vone}{\vskip 2ex}

\newcommand{\be}{\begin{equation}}
\newcommand{\ee}{\end{equation}}
\newcommand{\beano}{\begin{eqnarray*}}
\newcommand{\eeano}{\end{eqnarray*}}

\def\bmatrix#1{\left[ \begin{matrix} #1 \end{matrix} \right]}

\def\R{{\mathbb R}}
\def\C{{\mathbb C}}

\def \det{\mathrm{det}}
\def\lam{{\lambda}}

\title{Structural balance and spectral properties of generalized corona product of signed graphs}
\author{ Amrik Singh\thanks{Department of Electrical Engineering, IIT Jodhpur, India Email: amrik.iitj@gmail.com}, \,\, Ravi Srivastava\thanks{Department of Mathematics, NIT Sikkim, India Email: ravi@nitsikkim.ac.in}, \,\, Bibhas Adhikari\thanks{Department of Mathematics, IIT Kharagpur, India Email: bibhas@maths.iitkgp.ac.in}\, \thanks{The author currently works at Fujitsu Research of America, Inc., Sunnyvale, California, USA}, \,\, Sandeep Kumar Yadav\thanks{Department of Electrical Engineering, IIT Jodhpur, India Email: sy@iitj.ac.in }}
\date{}
\begin{document}
\maketitle

{\small \noindent{\bf Abstract.} In this paper, we extend our earlier proposal of corona product of signed graphs into generalized corona product of signed graphs inspired by the generalized corona product of unsigned graphs. Then we study structural balance and spectral properties of these graphs. Utilizing the notion of coronal of a graph, we determine computable formulae of characteristic, Laplacian, and signless Laplacian polynomials of generalized corona product of signed graphs. Finally, we provide sufficient conditions for the generalized corona product of some distinct collections of signed graphs to be co-spectral.


\vone \noindent{\bf Keywords.} generalized corona product, structural balance, Laplacian matrix, coronal 

\section{Introduction}\label{sec:1}
Corona product graphs are extremely useful for the proposal of complex network generative models for the generation of growing networks \cite{sharma2017structural,lv2015corona,sharma2019self,qi2018extended,wang2022modeling}. Recently, we have proposed the notion of corona product of two signed graphs as follows. Given two signed graphs, say $G^s_{\mu_{1}}=(V_{\mu_1},E^s_1,\sigma_1)$  and $H^s_{\mu_{2}}=(V_{\mu_2},E_2^s,\sigma_2),$ the corona product graph $G^s_{\mu}\circ H^s_{\mu}$ is defined by utilizing the framework of marked graphs \cite{bibhas2019signed}. Here $\sigma_j$ denotes the signature function, which assigns the signs $\pm$ (positive or negative) to the edges of the underlying (unsigned) graph, $j=1,2$. Marked graphs are signed graphs with a marking scheme that assigns signs to the graph's vertices. In this paper, we define and study generalized corona product of signed graphs by extending the notion of corona product of two signed graphs.

Recall that Cartwright and Harary first proposed the signed graph model to represent social relations by generalizing Heider’s theory of balanced cognitive unit \cite{cartwright1956structural}. The structural balance theory for signed graphs has recently been advanced in literature to model real-world phenomena. The edges with $\pm$ signs model conflicting binary relationships between data points. For instance, trust and distrust, friend and enmity, and agree or disagree are often modeled as $\pm$ edges in real-world graphs \cite{guha2004propagation,leskovec2010signed}. These models play a pivotal role in the analysis of salient features of the data points. 



The structural balance property of a signed graph is studied in terms of the number of balance cycles in the graph \cite{doreian2013brief}. A signed cycle graph is called a balanced cycle if it contains an even number of negative edges. Next, a signed graph is balanced if all its cycles (subgraphs) are balanced. In particular, cycles on three vertices, known as triads, also play a significant role in the structural balance theory. The four types of signed triads that can be present in a signed graph are exhibited in Figure \ref{fig:triangles}, where $T_j$ denotes the triad with $j$ negative edges, $j=0,1,2,3.$


 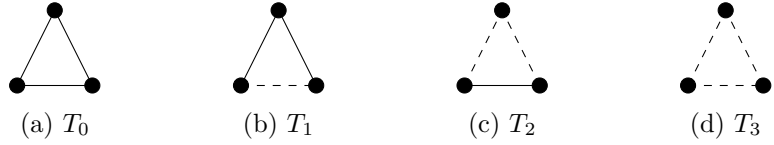
\begin{figure}
			\centering	
			\begin{subfigure}[b]{0.2\textwidth}
				\centering
				\begin{tikzpicture}
				\draw [fill] (0, 0) circle [radius=0.1];
				\draw [fill] (1, 0) circle [radius=0.1];
				\draw (0,0) --(1,0);
				\draw [fill] (0.5, 1) circle [radius=0.1];
				\draw (0,0) -- (0.5,1);
				\draw (1,0) -- (0.5,1);
				\end{tikzpicture}
				\caption{$T_0$}
			\end{subfigure}
			\begin{subfigure}[b]{0.2\textwidth}
				\centering
				\begin{tikzpicture}
				\draw [fill] (0, 0) circle [radius=0.1];
				\draw [fill] (1, 0) circle [radius=0.1];
				\draw [dashed] (0,0) --(1,0);
				\draw [fill] (0.5, 1) circle [radius=0.1];
				\draw (0,0) -- (0.5,1);
				\draw (1,0) -- (0.5,1);
				\end{tikzpicture}
				\caption{$T_1$}
			\end{subfigure}		
			\begin{subfigure}[b]{0.2\textwidth}
				\centering
				\begin{tikzpicture}
				\draw [fill] (0, 0) circle [radius=0.1];
				\draw [fill] (1, 0) circle [radius=0.1];
				\draw (0,0) --(1,0);
				\draw [fill] (0.5, 1) circle [radius=0.1];
				\draw [dashed] (0,0) -- (0.5,1);
				\draw [dashed] (1,0) -- (0.5,1);
				\end{tikzpicture}
				\caption{$T_2$}
			\end{subfigure}			
			\begin{subfigure}[b]{0.20\textwidth}
				\centering
				\begin{tikzpicture}
				\draw [fill] (0, 0) circle [radius=0.1];
				\draw [fill] (1, 0) circle [radius=0.1];
				\draw [dashed] (0,0) --(1,0);
				\draw [fill] (0.5, 1) circle [radius=0.1];
				\draw [dashed] (0,0) -- (0.5,1);
				\draw [dashed] (1,0) -- (0.5,1);
				\end{tikzpicture}
				\caption{$T_3$}
			\end{subfigure}
			\caption{Triads (a) and (c) are balanced, and (b) and (d) are unbalanced. Solid and dashed edges represent positive and negative edges, respectively.}
			\label{fig:triangles}
	\end{figure}

On the other hand, spectral properties of signed graphs are investigated in literature for several applications. The adjacency matrix, Laplacian matrix and signless Laplacian matrix associated with a signed graph on $n$ vertices are denoted by $A=[a_{ij}]$, $L=D-A$ and $L=D+A$, respectively. Here, $a_{ij}=1$ if the vertices $i$ and $j$ are adjacent through a $+$ edge, $a_{ij}=-1$ if the vertices $i$ and $j$ are adjacent through a $-$ edge, and $a_{ij}=0$ otherwise, $D=\mbox{diag}\{d(1), \hdots, d(n)\}$, $d(j)=d^+(j) + d^-(j)$, where  $d^+(j) \text{ and } d^-(j)$ are the number of positive and negative edges incident to the node $j$, respectively. The net degree of a node $j$ in a signed graph $G$ is given by $d^+(j) - d^-(j)$. A signed graph $G$ is called \textit{net-regular} if each node has the same net degree $d$ \cite{netregular,shahul2015co}. Structural balance of a signed graph is often measured by \textit{algebraic conflict}, which is the smallest signed Laplacian eigenvalue of the graph. Besides, adjacency and Laplacian eigenvalues of a signed graph determine the converging state of diffusion dynamics defined on a signed graph, which represents a multi-agent system \cite{yang2021bipartite}. Two signed graphs are considered co-spectral when their signed characteristic polynomials match. This concept is further categorized as signed L-cospectral and signed Q-cospectral when the characteristic polynomials involved are signed Laplacian and signed signless Laplacian, respectively.


Finding analytical expression of eigenvalues and Laplacian eigenvalues of product graphs such as corona product graphs is one of the fundamental problems in spectral graph theory \cite{sharma2015spectra,mcleman2011spectra}. The technique of writing the characteristic polynomial $\left| A(G\circ H)-\lam I\right|$ and the $L$-polynomial $\left| L(G\circ H)-\lam I\right|$ of corona product graph $G\circ H$ into a product of polynomials associated with $G, H$ is established by considering a new graph invariant known as \textit{coronal} of a graph \cite{mcleman2011spectra}. It is shown that $f_{G\circ H}(\lam) = f_H(\lam)^m \left(f_G(\lam) - \chi_H(\lam)\right),$ where $f_{G\circ H}, f_G, f_H$ denote the characteristic polynomial of $G\circ H, G$ and $H$ respectively, $m$ is the number of vertices of $H$, and $\chi_H$ denotes the coronal of the graph $H$, defined as the sum of all the entries of $(\lam I - A(H))^{-1}.$ Further, coronal is used to factorize characteristic polynomial of generalized corona product of graphs in \cite{laali2016spectra}. An obvious accomplishment of the notion of coronal is that eigenvalues of corona product of two graphs can be obtained in terms of eigenvalues of the constituent graphs which define the corona product. We denote the Laplacian polynomial and the signless Laplacian polynomial of a graph $G$  by $f_{L(G)}(\alpha)$ and $f_{Q(G)}(\beta)$ respectively. 


    In this paper, we have extended the definition of the generalized corona product from unsigned graphs \cite{laali2016spectra} to signed graphs and have explored their structural and spectral properties. We introduced the concepts of signed coronal, Laplacian signed coronal (L-signed coronal), and Signless Laplacian signed coronal (Q-signed coronal) for signed graphs by building upon the concept of 'coronal' introduced by McLeman and McNicholas \cite{mcleman2011spectra} as well as 'Laplacian Coronal' and 'Signless Laplacian coronal' introduced by Laali et al. \cite{laali2016spectra}. Furthermore, we conducted a study on the spectral properties of generalized corona products applied to signed graphs, including the derivation of the signed characteristic polynomial, signed Laplacian polynomial, and signed signless Laplacian polynomial. Additionally, we introduced the concept of 'signed marked graph isomorphism,' enabling us to establish sufficient conditions for the generalized product of distinct collections of signed graphs to exhibit co-spectral behaviour. 

    The rest of the paper is organized as follows. In Section \ref{sec:2}, we introduce generalized corona product of signed graphs and study their structural balance properties. We investigate the spectral properties of these graphs in Section \ref{sec:3}. Then we conclude the paper.

 \section{Generalized corona product of signed graphs and structural balance properties} \label{sec:2}
 
    A signed graph is given by a triple $G^s = (V, E^s, \sigma)$, where $V$ is a finite set of vertices, $E^s\subset V\times V$ is a set of edges, and $\sigma \colon E^s \longrightarrow \{ +,-\}$ is a signature function. Then the canonical marking scheme $\mu: V\rightarrow \{1, -1\}$ on $G^s$ is defined as 

    \begin{equation}\label{Def:markingofanode}\
        \mu(v)=\prod_{e\in E^s} \sigma(e)
    \end{equation}
    where $E^s$ is the set of signed edges adjacent at $v$  \cite{beineke1978consistent,beineke1978consistent2}. Throughout the paper, we consider a signed graph as a marked graph with the canonical marking scheme, and we denote such a graph as a quadruple $G^s_\mu=(V_\mu, E^s,\sigma, \mu).$ Now we recall the definition of corona product of two signed graphs as follows. 

    \begin{definition}\label{def:cpg}\cite{bibhas2019signed}
        Let $G^s_{\mu_1}=(V_{\mu_1},E^s_1,\sigma_1,\mu_1)$  and $H^s_{\mu_2}=(V_{\mu_2},E_2^s,\sigma_2,\mu_2)$ be signed graphs on $n$ and $k$ nodes respectively. Then corona product $G^s_{\mu_1}\circ H^s_{\mu_2}$ of $G^s_{\mu_1}, H^s_{\mu_2}$ is a signed graph by taking one copy of $G^s_{\mu_1}$ and $n$ copies of $H^s_{\mu_2},$ and then forming a signed edge from $i^{th}$ node of $G^s_{\mu_1}$ to every node of the $i^{th}$ copy of $H^s_{\mu_2}$ for all $i.$ The sign of the new edge between $i^{th}$ node of $G^s_{\mu_1},$ say $u$ and $j^{th}$ node in the $i^{th}$ copy of $H^s_{\mu_2},$ say $v$ is given by $\mu_1(u)\mu_2(v)$ where $\mu_1$ and $\mu_2$ are a marking schemes defined by $\sigma_1$, $\sigma_2$ on $G^s_{\mu_1}$ and $H^s_{\mu_2}$ respectively.
     \end{definition}
    Based on the notion of corona of two unsigned graphs, the generalized corona product of two unsigned graphs is introduced in \cite{laali2016spectra}. Now we extend this definition of generalized corona product of unsigned graphs to signed graphs with the help of Definition \ref{def:cpg} as follows. 
    \begin{definition}\label{def:cpg2}
        \begin{sloppypar}
        Let $G^s_{\mu}=(V_{\mu},E^s,\sigma,\mu)$  be a signed graph on $n$ vertices. Let $H^s_{\mu_l}=(V_{\mu_l},E_l^s,\sigma_l,\mu_l)$ be signed graphs  with $|V_{\mu_l}|=k_l$, $1\le l\le n$. Then, the generalized corona product of $G^s_{\mu}$ and $H^s_{\mu_l},$ denoted by $$G^s_{\mu}\circ \overset{n}{\underset{l=1}{ \Lambda}}H^s_{\mu_l},$$ is obtained when the $l$-th vertex of $G^s_{\mu}$ is linked with every node of $H^s_{\mu_l}.$ If $u$ is the $l$-th vertex of $G^s_{\mu}$ and $v$ be any vertex of $H^s_{\mu_l}$ then the sign of the edge joining $u$ and $v$ is given by $\mu(u)\mu_{l}(v).$ 
        \end{sloppypar}
    \end{definition}

    For example, see Figure \ref{fig:triads}. In what follows, we provide the statistics of signed edges and balanced/unbalanced triads in generalized signed corona graphs. We also determine necessary and sufficient conditions for a balanced generalized corona product of signed graphs in terms of patterns of signs of edges adjacent to certain types of marked vertices of the constituent signed graphs. 
    


    Keeping the notations in Definition \ref{def:cpg2}, it is easy to verify that the number of nodes in $G^s_\mu \circ \overset{n}{\underset{l=1}{\Lambda}} H^s_{\mu_l}$ is equal to $|V_\mu|+\sum_{l=1}^{n}|V_{\mu_l}|$ where $|V_G|$ and $|V_{H^s_{\mu_l}}|$ are the number of nodes in $G^s_\mu$ and $H^s_{\mu_l}$, respectively, $l = 1,2,\dots,n.$ Let $M_{G^s_{\mu}}^+$ ($M_{G^s_{\mu}}^-$) and $M_{H^s_{\mu_l}}^{+}$ ($M_{H^s_{\mu_l}}^{-}$) denote the number of positively (negatively) marked nodes in $G^s_\mu$ and $H^s_{\mu_l}$, respectively. Then Table \ref{tb:edges_stat} lists the number of signed edges  in $G^s_\mu \circ \overset{n}{\underset{l=1}{\Lambda}}H^s_{\mu_l}$ that can be derived easily, where  $|E_{l}^{s}|^{\overset{+}{+}}$ ($|E^{s}|^{\overset{+}{+}}$) is the number of edges with sign $s\in\{+, -\}$ that connect positively marked nodes, $|E_{l}^{s}|^{\pm}$ ($|E^{s}|^{\pm}$) is the number of edges with sign $s$ that connect one positively marked node and one negatively marked node, and $|E_{l}^{s}|^{\overset{-}{-}}$ ($|E^{s}|^{\overset{-}{-}}$) is the number of edges with sign $s$ that connect negatively marked nodes in $H^s_{\mu_l}$ ($G^s_\mu$), where $s\in\{+, -\}$.
    	\begin{table}[!]
    		\centering
    			\begin{tabular}{|c|c|c|c|}
    				\hline 
    				Edges & $G^s_\mu$  & $H^s_{\mu_l}$ & $G^s_\mu \circ \overset{n}{\underset{l=1}{\Lambda}}H^s_{\mu_l}$ \\ 
    				\hline \hline
    				\# of edges & $|E^s|$ & $|E^s_{\mu_l}|$ & $|E^s| + \overset{n}{\underset{l=1}{\sum}}|E^s_l|$ + $\overset{n}{\underset{l=1}{\sum}}|V_{\mu_l}|$\\ 
    				\hline 
    				\# of$+$  edges & $|E^{+}|$ & $|E_{\mu_l}^{+}|$ & $|E^{+}| + \overset{n}{\underset{l=1}{\sum}}|E_l^{+}|$ + $M_{G^s_\mu}^+ \overset{n}{\underset{l=1}{\sum}} M_{H^s_{\mu_l}}^{+}$ + $M_{G^s_\mu}^- \overset{n}{\underset{l=1}{\sum}} M_{H^s_{\mu_l}}^{-}$ \\ 
    				\hline 
    				\# of$-$  edges & $|E^{-}|$ & $|E_l^{-}|$ & $|E^{-}| + \overset{n}{\underset{l=1}{\sum}}|E_l^{-}|$ + $M_{G^s_\mu}^+ \overset{n}{\underset{l=1}{\sum}} M_{H^s_{\mu_l}}^{-}$ + $M_{G^s_\mu}^- \overset{n}{\underset{l=1}{\sum}} M_{H^s_{\mu_l}}^{+}$ \\  
    				\hline 		
    			\end{tabular} 
    		\caption{Statistics of number of signed edges  in $G^s_\mu \circ \overset{n}{\underset{l=1}{\Lambda}}H^s_{\mu_l}$.}
    		\label{tb:edges_stat}
    	\end{table}

    Now we consider counting triads of each type in $G^s_\mu \circ \overset{n}{\underset{l=1}{\Lambda}}H^s_{\mu_l}$. Let $|T_i(G^s_{\mu})|$ and $|T_i(H^s_{\mu_l})|$ denote the number of triads, having $i$ number of negative edges, in the graph $G^s_{\mu}$ and $H^s_{\mu_l}$, respectively, for $l=1,2,\dots,n$. Then in table \ref{tb:triads_stat}, we list the number of triads in $G^s_{\mu}\circ \overset{n}{\underset{l=1}{ \Lambda}}H^s_{\mu_l}$ and can be proved easily. Indeed, note that in addition to the triads present in $G^s_{\mu}$ and $H^s_{\mu_l}$, new triads are created due to the definition of corona product. A triad is formed when a node $l$ in $V_\mu$ is linked to an edge $(j,k)$ in $E^s_{{\mu_l}}$ and the type of this triad depends on the marking of nodes $(l,j, \text{ and }k)$ and the sign of the edge $(j,k)$. Thus, $|E_{H^s_{\mu_l}}|$ new triads are created by linking node $l \in V_\mu$ with all the nodes in $V_{\mu_l}$, for $l=1,2,\dots,n$. Hence, the total number of triads in $G^s_{\mu}\circ \overset{n}{\underset{l=1}{ \Lambda}}H^s_{\mu_l}$ is 
    	\begin{equation}
    		T(G^s_{\mu}\circ \overset{n}{\underset{l=1}{ \Lambda}}H^s_{\mu_l}) =
    			|T(G^s_{\mu})| + \overset{n}{\underset{l=1}{\sum}}|T(H^s_{\mu_l})| + \overset{n}{\underset{l=1}{\sum}}|E^s_l|.
    	\end{equation}  
     \begin{sloppypar}
    	\begin{table}[!]
    		\centering                
    		\begin{tabular}{|c|c|c|c|}
    			\hline 
    			triads & $G^s_{\mu}$  & $H^s_{\mu_l}$ & $G^s_{\mu}\circ \overset{n}{\underset{l=1}{ \Lambda}}H^s_{\mu_l}$\\ 
    			\hline \hline
    			\rule[-1ex]{0pt}{2.5ex} \# of $T_0$ & $|T_0(G^s_{\mu})|$ & $|T_0(H^s_{\mu_l})|$ & $|T_0(G^s_{\mu})| +  \overset{n}{\underset{l=1}{\sum}}|T_0(H^s_{\mu_l})| + M_{G^s_{\mu}}^+ \overset{n}{\underset{l=1}{\sum}} |E_{l}^{s+}|^{\overset{+}{+}} + M_{G^s_{\mu}}^- \overset{n}{\underset{l=1}{\sum}} |E_{l}^{s+}|^{\overset{-}{-}}$\\ 
    			\hline 
    			\rule[-1ex]{0pt}{2.5ex}  \# of $T_1$ & $|T_1(G^s_{\mu})|$ & $|T_1(H^s_{\mu_l})|$ & $|T_1(G^s_{\mu})|$ +  $\overset{n}{\underset{l=1}{\sum}}|T_1(H^s_{\mu_l})| + M_{G^s_{\mu}}^+ \overset{n}{\underset{l=1}{\sum}} |E_{H^s_{\mu_l}}^{+}|^{\pm} + M_{G^s_{\mu}}^+ \overset{n}{\underset{l=1}{\sum}} |E_{H^s_{\mu_l}}^{-}|^{\overset{+}{+}}$ \\
    			& & & $+ M_{G^s_{\mu}}^- \overset{n}{\underset{l=1}{\sum}} |E_{H^s_{\mu_l}}^{+}|^{\pm} + M_{G^s_{\mu}}^- \overset{n}{\underset{l=1}{\sum}} |E_{H^s_{\mu_l}}^{-}|^{\overset{-}{-}}$ \\ 
    			\hline 
    			\rule[-1ex]{0pt}{2.5ex}  \# of $T_2$ & $|T_2(G^s_{\mu})|$ & $|T_2(H^s_{\mu_l})|$ & $|T_2(G^s_{\mu})| $ +  $\overset{n}{\underset{l=1}{\sum}}|T_2(H^s_{\mu_l})| + M_{G^s_{\mu}}^+ \overset{n}{\underset{l=1}{\sum}} |E_{H^s_{\mu_l}}^{+}|^{\overset{-}{-}} + M_{G^s_{\mu}}^+ \overset{n}{\underset{l=1}{\sum}} |E_{H^s_{\mu_l}}^{-}|^{\pm} $\\
    			& & & $ + M_{G^s_{\mu}}^- \overset{n}{\underset{l=1}{\sum}} |E_{H^s_{\mu_l}}^{+}|^{\overset{+}{+}} + M_{G^s_{\mu}}^- \overset{n}{\underset{l=1}{\sum}} |E_{H^s_{\mu_l}}^{-}|^{\pm}$ \\ 
    			\hline 
    			\rule[-1ex]{0pt}{2.5ex} \# of $T_3$ & $|T_3(G^s_{\mu})|$ & $|T_3(H^s_{\mu_l})|$ & $|T_3(G^s_{\mu})| +  \overset{n}{\underset{l=1}{\sum}}|T_3(H^s_{\mu_l})| + M_{G^s_{\mu}}^+ \overset{n}{\underset{l=1}{\sum}} |E_{H^s_{\mu_l}}^{-}|^{\overset{-}{-}} + M_{G^s_{\mu}}^- \overset{n}{\underset{l=1}{\sum}} |E_{H^s_{\mu_l}}^{-}|^{\overset{+}{+}} $\\  
    			\hline 
    		\end{tabular}             
    		\caption{Count of triads in $G^s_{\mu}\circ \overset{n}{\underset{l=1}{ \Lambda}}H^s_{\mu_l}$.}
    		\label{tb:triads_stat}
    	\end{table}
     \end{sloppypar}

    Towards investigating the structural balance properties of the generalized corona product of signed graphs, first, note that $G^s_{\mu}\circ \overset{n}{\underset{l=1}{ \Lambda}}H^s_{\mu_l}$, $l=1,2,\dots,n$ is unbalanced if any one graph among $G^s_{\mu}, H^s_{\mu_1}, H^s_{\mu_2},\dots, 
    H^s_{\mu_l}$ is unbalanced. Thus in \cite{bibhas2019signed},  conditions are derived under which corona product of balanced signed graphs becomes unbalanced. We extend this result for  $G^s_{\mu}\circ \overset{n}{\underset{l=1}{ \Lambda}}H^s_{\mu_l}$, $l=1,2,\dots,n$ to be unbalanced when all the constituent signed graphs are balanced.

	\begin{theorem}
		Let $G^s_\mu=(V_{\mu},E^s,\sigma,\mu)$ and $H^s_{\mu_l}=(V_{\mu_l},E_l^s,\sigma_l,\mu_l)$ be structurally balance signed graphs, $l=1,2,\dots,n$. Then  $G^s_{\mu}\circ \overset{n}{\underset{l=1}{ \Lambda}}H^s_{\mu_l}$ is unbalanced if there exists at least one $H^s_{\mu_j}$, $1 \le j \le n$, having one of the following types of edges
		\begin{enumerate}
			\item a positive edge that connects two opposite marked nodes,
			\item a negative edge that connects two negatively marked nodes, and
			\item a negative edge that connects two positively marked nodes.
		\end{enumerate}
	\end{theorem}
	\begin{proof}
	      The proof is similar to Theorem 2.2 in                \cite{bibhas2019signed}. A positively marked node in $G^s_{\mu}$ will form a triad $T_1$ in $G^s_{\mu}\circ \overset{2}{\underset{l=1}{ \Lambda}}H^s_{\mu_l}$ when an edge of type $1$ and $3$ is present in $H^s_{\mu_j}$, otherwise it forms a triad $T_3$ when an edge of type $2$ is present in $H^s_{\mu_j}$.  Similarly, a negatively marked node in $G^s_{\mu}$ will form a triad $T_1$ in $G^s_{\mu}\circ \overset{2}{\underset{l=1}{ \Lambda}}H^s_{\mu_l}$ when an edge of type $1$ and $2$ is present in $H^s_{\mu_j}$, otherwise it forms a triad $T_3$ when an edge of type $3$ is present in $H^s_{\mu_j}$. 
	\end{proof}
	
    In Figure \ref{fig:coronaproduct}, we give an example of an unbalanced generalized corona product graph of structurally balanced signed graphs  $G^s_{\mu}, H^s_{\mu_1}, \text{ and } H^s_{\mu_2}$.
    	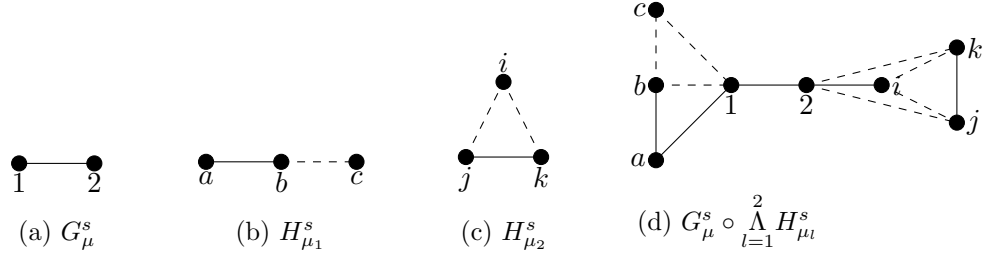
\begin{figure}
    		\centering	
    		\begin{subfigure}[b]{0.2\textwidth}
    			\centering
    			\begin{tikzpicture}
    			\draw [fill] (0, 0) circle [radius=0.1];
    			\draw [fill] (1, 0) circle [radius=0.1];
    			\draw (0,0) --(1,0);	
    			\node [below] at (0,0) {$1$};
    			\node [below] at (1,0) {$2$};		
    			\end{tikzpicture}
    			\caption{$G^s_\mu$}
    		\end{subfigure}
    		\begin{subfigure}[b]{0.2\textwidth}
    			\centering
    			\begin{tikzpicture}
    			\draw [fill] (0, 0) circle [radius=0.1];
    			\draw [fill] (1, 0) circle [radius=0.1];			
    			\draw [fill] (2,0) circle [radius=0.1];
    			\draw (0,0) --(1,0);
    			\draw [dashed] (1,0) -- (2,0);
    			\node [below] at (0,0) {$a$};
    			\node [below] at (1,0) {$b$};
    			\node [below] at (2,0) {$c$};
    			\end{tikzpicture}
    			\caption{$H^s_{\mu_1}$}
    		\end{subfigure}						
    		\begin{subfigure}[b]{0.20\textwidth}
    			\centering
    			\begin{tikzpicture}
    			\draw [fill] (0, 0) circle [radius=0.1];
    			\draw [fill] (1, 0) circle [radius=0.1];
    			\draw [fill] (0.5, 1) circle [radius=0.1];
    			\draw (0,0) --(1,0);			
    			\draw [dashed] (0,0) -- (0.5,1);
    			\draw [dashed] (1,0) -- (0.5,1);
    			\node [below] at (0,0) {$j$};
    			\node [below] at (1,0) {$k$};
    			\node [above] at (0.5,1) {$i$};
    			\end{tikzpicture}
    			\caption{$H^s_{\mu_2}$}
    		\end{subfigure}
    		\begin{subfigure}[b]{0.20\textwidth}
    			\centering
    			\begin{tikzpicture}
    			\draw [fill] (0,0) circle [radius=0.1];
    			\draw [fill] (0,1) circle [radius=0.1];
    			\draw [fill] (0,2) circle [radius=0.1];
    			\draw [fill] (1,1) circle [radius=0.1];
    			\draw [fill] (2,1) circle [radius=0.1];
    			\draw [fill] (3,1) circle [radius=0.1];
    			\draw [fill] (4,0.5) circle [radius=0.1];
    			\draw [fill] (4,1.5) circle [radius=0.1];
    			
    			\node [left] at (0,0) {$a$};
    			\node [left] at (0,1) {$b$};
    			\node [left] at (0,2) {$c$};
    			\node [below] at (1,1) {$1$};
    			\node [below] at (2,1) {$2$};
    			\node [right] at (3,1) {$i$};
    			\node [right] at (4,0.5) {$j$};
    			\node [right] at (4,1.5) {$k$};
    			
    			\draw (0,0) --(0,1);			
    			\draw [dashed] (0,1) -- (0,2);
    			\draw (1,1) -- (2,1);
    			\draw [dashed] (3,1) -- (4,0.5);
    			\draw [dashed] (3,1) -- (4,1.5);
    			\draw (4,1.5) -- (4,0.5);
    			
    			\draw (0,0) --(1,1);
    			\draw [dashed] (0,1) -- (1,1);
    			\draw [dashed] (0,2) -- (1,1);			
    			\draw (2,1) -- (3,1);
    			\draw [dashed] (2,1) -- (4,0.5);
    			\draw [dashed] (2,1) -- (4,1.5);
    			\end{tikzpicture}
    			\caption{$G^s_{\mu}\circ \overset{2}{\underset{l=1}{ \Lambda}}H^s_{\mu_l}$}
    			\label{fig:coronaproduct}
    		\end{subfigure}
    		\caption{The generalized corona product $G^s_{\mu}\circ \overset{2}{\underset{l=1}{ \Lambda}}H^s_{\mu_l}$ is shown in (d) using canonical marking scheme.}
    		\label{fig:triads}
    	\end{figure}
\section{Spectral properties of generalized corona product of graphs} \label{sec:3}

    In this section we define the coronal of a signed graph by generalizing the notion of coronal of an unsigned graph, which is a graph invariant introduced in \cite{mcleman2011spectra}. Further, we determine the characteristic polynomials for the adjacency matrix, Laplacian matrix, and Signless Laplacian matrix, we named Characteristic polynomials, signed Laplacian polynomials, and signed signless Laplacian polynomials respectively corresponding to the generalized corona product of signed graphs in terms of certain polynomials corresponding to the constituent signed graphs which define the product. Additionally, we present criteria for generalised corona product signed graphs to exhibit cospectral properties, including $L$-cospectrality and $Q$-cospectrality.

   
    
    \begin{definition} \cite{mcleman2011spectra}
        Assume that H is a graph on n vertices with the adjacency matrix B, when considered as a matrix over the field of rational function $\C(\lambda)$, the characteristic matrix $\lambda I - B$ has determinant  $\det(\lambda I - B) = f_H(\lambda) \neq 0$, indicating it is inevitable. The coronal $\chi_H(\lambda) \in \C(\lambda)$ of $H$ is defined to be the sum of the entries of the matrix $(\lambda I - B)^{-1}$. Therefore, this can be obtained as $$\chi_H(\lambda)=\textbf{1}^T(\lambda I -B)^{-1}\textbf{1}$$ where $\textbf{1}$ denotes the all-one vector of compatible dimension, that is $n$.
     \end{definition}
    

   Now we introduce the notion of isomorphism for signed marked graphs. Let $G^s_\mu=(V_\mu, E^s, \sigma)$ be a signed marked graph with the marking scheme $\mu: V_\mu \rightarrow \{-1, +1\}$, and the signature function $\sigma: E^s \rightarrow \{-1, +1\}.$ As mentioned before, the marking scheme $\mu$ is considered the canonical marking scheme for the signed graph $G^s_{\mu}=(V, E^s, \sigma)$.     
    \begin{definition} 
        Given two signed marked graphs $G^s_{\mu_1}=(V_{\mu_1}, E^s_1, \sigma_1)$ and $G^s_{\mu_2}=(V_{\mu_2}, E^s_2, \sigma_2)$, we call a bijective map $\phi : V_{\mu_1} \rightarrow V_{\mu_2}$ isomorphism if 
        \begin{enumerate}
            \item $(i, j)\in E^s_1$ if and only if $(\phi(i), \phi(j)) \in E^s_2,$ and $\sigma_1((i,j))=\sigma_2(\phi(i),\phi(j))$ for all $(i,j)\in E^s_1$            
             \item $\mu_1[V_{\mu_1}] =\mu_2[\phi(V_{\mu_1})],$ where 
                \beano 
                \mu_1[V_{\mu_1}] &=& [\mu_1(v_1), \mu_1(v_2), \, \hdots, \mu_1(v_k)]^T, V_{\mu_1}=\{v_1, \hdots, v_k\} \,\, \mbox{and} \\
                \mu_2[\phi(V_{\mu_1})] &=& [\mu_2
                (\phi(v_1)), \mu_2(\phi(v_2)), \, \hdots, \mu_2(\phi(v_k))]^T.
                \eeano 
        \end{enumerate}
    \end{definition}    
    Obviously, the isomorphic signed graphs $G_{\mu_j}^s=(V_{\mu_j}, E^s,\sigma_j),$ $j=1,2$ have the same characteristic polynomials. Indeed, if $\phi$ denotes the isomorphism then  $A( G^s_{\mu_1})=P_\phi A(G^s_{\mu_2}) P_\phi^T$ and $\mu[V_{\mu_1}]=P_\phi \mu[V_{\mu_2}],$ where $A(G_{\mu_j}^s)$ denotes the adjacency matrix of $G_{\mu_j}^s$ and $P_\phi$ is the permutation matrix corresponding to the bijective function $\phi.$ 
    
    Let $G^s_\mu=(V_{\mu},E^s,\sigma)$  be a signed graph with vertex set $V_\mu=\{v_1, v_2,\dots, v_n\}$ and adjacency matrix $A(G^s_\mu)$. Suppose for $l=1,2,\dots,n$, $H^s_{\mu_l}=(V_{\mu_l},E_l^s,\sigma_l)$ be a signed graph with $|V_{\mu_l}|=t_l$, adjacency matrix $B_l$ and vertex set  $V{\mu_l}=\{ n+\sum_{k=1}^{l-1}{t_k}+1, n+\sum_{k=1}^{l-1}{t_k}+2,\dots,n+\sum_{k=1}^l{t_k}\}$. The adjacency matrix of $G^s_{\mu}\circ \overset{n}{\underset{l=1}{ \Lambda}}H^s_{\mu_l}$ is presented with this labeling as follows:\\    	
    $$ A\left(G^s_{\mu}\circ \overset{n}{\underset{l=1}{ \Lambda}}H^s_{\mu_l}\right)= \left[ \begin{matrix}
			 A(G^s_{\mu}) & PQ \\ \\
			(PQ)^T & D
			\end{matrix}
			\right]$$
		where \begin{equation}\label{eqn:pq} P = \left[ \begin{matrix}
		\mu(v_1) & 0 & 0& \dots & 0 \\
		0& \mu(v_2) & 0 & \dots & 0 \\
		\vdots & \vdots & \ddots & \vdots \\
    	0  & 0 & 0 & \dots & \mu(v_n)
    		\end{matrix} \right], \,\, 
    Q = \left[ \begin{matrix}
		\mu_1[V_{\mu_1}]^T & 0 & 0& \dots & 0 \\
		0& \mu_2[V_{\mu_2}]^T & 0 & \dots & 0 \\
		\vdots & \vdots & \ddots & \vdots \\
    	0  & 0 & 0 & \dots & \mu_n[V_{\mu_n}]^T
    		\end{matrix} \right],
    	\end{equation}
    	and 
    	\begin{equation}\label{eqn:d} D=
    \begin{blockarray}{cccccc}
    &H_1 & H_2 & H_3 & \dots & H_n \\
    \begin{block}{c[ccccc]}
      H_1 & B_1 & 0 & 0 & \dots & 0 \\
      H_2 & 0 & B_2 & 0 & \dots & 0 \\
      \vdots & \vdots & \vdots & \ddots & \vdots& \vdots \\
      H_{n-1} & 0 & 0 & \dots & B_{n-1} & 0 \\
      H_n & 0 & 0 & \dots & 0 & B_n \\
    \end{block}
    \end{blockarray}
    \end{equation}
    Consequently 
    $ P\in\R^{n \times n}$, $Q\in \R^{n \times (\sum_{l=1}^nt_l)}$ and  $D \in R^{(\sum_{l=1}^nt_l) \times (\sum_{l=1}^nt_l)}$.	
    Now we recall the following, which will be used in the sequel.
    
    \begin{lemma}(Schur Complement). (See \cite{horn2012matrix}.)\label{schur_complement}
            Suppose $A$ is a $n\times n$ matrix with partitioned form $\left[ \begin{matrix}
    			A_{11} & A_{12} \\
    			A_{21} & A_{22}
    			\end{matrix}
    			\right]$
    			where $A_{11}$ and $A_{22}$ are inevitable matrices. Then 
    	\begin{equation*} \label{eq1}
                \begin{split}
                \det\left[ \begin{matrix}
                			A_{11} & A_{12} \\
                			A_{21} & A_{22}
                			\end{matrix}
                			\right]=&\det(A_{22}) \, \det(A_{11}-A_{12}{A_{22}}^{-1}A_{21}) \\
                 =&\det(A_{11}) \, \det(A_{22}-A_{21}{A_{11}}^{-1}A_{12}). 
                \end{split}
            \end{equation*}	
    \end{lemma}

    For a signed graph $H^s_{\mu } = (V_{\mu},E^s,\sigma)$, we define signed coronal and $L$-signed coronal to be the scalars 
    $${\chi}_{H^s_{\mu}}(\lambda)  =  \mu \left[ V_\mu \right] ^T(\lambda I- A({H^s_{\mu}}))^{-1}\mu \left[V_\mu\right]$$ 
    and 
    $${\chi}_{L({H^s_\mu})}(\alpha)  = \mu \left[ V_\mu\right]^T((\alpha-1)I-L(H^s_{\mu}))^{-1} \mu \left[ V_\mu\right],$$ respectively. Similarly, we define $Q$-signed coronal of the signed  graph $H^s_{\mu}$ as the scalar $${\chi}_{Q(H^s_{\mu})}(\beta) =  \mu[V_\mu]^T((\beta-1)I-Q(H^s_{\mu}))^{-1} \mu[V_\mu].$$
        
    In the following proposition, we show that the coronal of a marked signed graph is preserved under isomorphism.
    \begin{proposition} {\label{lm:isomorphic_equalcoronal}}
        If $H^s_{\mu_1}=(V_{\mu_1}, E^s, {\sigma_1})$ and $Z^s_{\mu_2}=(V_{\mu_2}, E^s, {\sigma_2})$ are two isomorphic signed marked graphs then $$\chi_{H^s_{\mu_1}}(\lambda)=\chi_{H^s_{\mu_2}}(\lambda).$$\end{proposition}
        \begin{proof} 
            Since $H^s_{\mu_1}$ and $Z^s_{\mu_2}$ are isomorphic, therefore there exist a permutation matrix $P_{\phi}$ such that $$A({Z^s_{\mu_2}})=P_{\phi}A({H^s_{\mu_1}})P_{\phi}^T \quad \mbox{and} \quad \mu_{2}[V_{\mu_2}]=P_\phi \mu_{1}[V_{\mu_1}].$$    
            Now, 
            \begin{equation*}\label{eq3}
                \begin{split}
                    \chi_{Z^s_{\mu_2}}(\lambda) &= \mu_2\left[ V_{\mu_2}\right]^T\left( \lambda I-A({Z^s_{\mu_2}})\right) ^{-1}\mu_2 \left[V_{\mu_2}\right] \\
                 &= \mu_1\left[V_{\mu_1}\right]^T P_{\phi}^T\left( \lambda I-P_{\phi}A({H^s_{\mu_1})P_{\phi}^T}\right) ^{-1}P_\phi\mu_1 \left[V_{\mu_1}\right] \\
                 &=\mu_1\left[V_{\mu_1}\right]^T P_{\phi}^T P_{\phi}^T\left( \lambda I-A({H^s_{\mu_1})}\right) ^{-1}P_\phi P_\phi\mu_1 \left[V_{\mu_1}\right]\\
                 &= \mu_1\left[V_{\mu_1}\right]^T\left( \lambda I-A({H^s_{\mu_1}})\right) ^{-1}\mu_1 \left[V_{\mu_1}\right]\\
                 &= \chi_{H^s_{\mu_1}}(\lambda). \qedhere
                \end{split}
            \end{equation*}
            \end{proof}     
    A signed graph $H^s_{\mu_{2}}=(V_{\mu_2},E_2^s,\sigma_2)$ is called $k$-net regular if $d^+(v) - d^-(v) = k$, for each $v \in V_{{\mu}_{2}}$, where $d^+(v)$ and $d^-(v)$ are the positive and negative degree of node $v$, respectively \cite{zaslavsky2010matrices}. \\ \\
    In the subsequent propositions, we establish explicit formulas to determine the signed coronal, L-signed coronal, and Q-signed coronal, focusing on particular signed graph instances.
    
    \begin{proposition}
       {\label{formula_signedcoronal}}
        Let $H^s_{\mu}$ be a $k$-net regular signed graph of order $n$. Then $$\chi_{H^s_{\mu}}(\lambda)=\frac{n}{\lambda-k}.$$
    \end{proposition}
    \begin{proof} 
        Since $H^s_{\mu}$ is a $k$-net regular signed graph, we have $\mu[V_\mu] = \textbf{1}_n \, \text{or} \, -\textbf{1}_n $   so in either case $A({H^s_{\mu }})\mu [V_\mu] = k \mu [V_\mu ]$. This implies that $(\lambda I-A({H^s_{\mu}})) \mu[V_\mu]=(\lambda-k) \mu [V_\mu]$. Therefore
      \[
           \chi_{H^s_{\mu}}(\lambda) =\mu [V_\mu]^T\left( \lambda I-A({H^s_{\mu}})\right) ^{-1}\mu [V_\mu] 
       =\frac{\mu [V_\mu ]^T \mu [V_\mu ]}{\lambda-k}=\frac{n}{\lambda-k} \qedhere
       \]
    \end{proof}
    A signed graph is defined as co-regular if it is net-regular with net-degree $k$ for some integer $k$ and its underlying graph is $r$-regular for some positive integer $r$. In this case, we define the ordered pair $(r, k)$ as the co-regularity pair \cite{hameed2012on}.    
    \begin{proposition}{\label{lm:Laplacian_coregular}}
            Let $H^s_{\mu}$ be a co-regular signed graph of co-regularity pair $(r,k)$. Then
            $${\chi}_{L({H^s_\mu})}(\alpha) = \frac{n}{\alpha-1-2d^{-}},$$
            where $d^{-}$ is the negative degree of any vertex of $H^s_{\mu}$.
    \end{proposition}          
    \begin{proof}
         Since $H^s_{\mu}$ is a co-regular signed graph, each node has the same net degree, and the sum of the entries in each row of its Laplacian matrix is equal to two times the negative degree of any vertex of $H^s_{\mu}$ so we have $L(H^s_{\mu})\mu[V_\mu]=2d^{-}\mu[V_\mu]$, where $\mu[V_\mu] = \textbf{1}_n \text{ or } -\textbf{1}_n $ and in either case  
    
         \begin{align*}
            \left((\alpha-1)I-L(H^s_{\mu})\right).\mu [V_\mu] &= \left((\alpha -1)-2d^{-}\right)I \mu [V_\mu] \\
             &=\left(\alpha -1-2d^{-}\right)\mu [V_\mu].\\
        \end{align*}
       
        This implies $\mu [V_\mu]^T\left((\alpha-1)I-L(H^s_{\mu})\right)^{-1}\mu [V_\mu] = \frac{\mu [V_\mu]^T\mu [V_\mu]}{\alpha -1-2d^{-}}.$ Therefore the desired result follows.
    \end{proof}   
    \begin{proposition}  {\label{proposition 3.5}}
          Let $H^s_{\mu}$ be a co-regular signed graph of order $n$ and co-regularity pair $(r,k)$. Then
         $${\chi}_{Q(H^s_{\mu})}(\beta) = \frac{n}{\beta-1-2d^{+}},$$
         where $d^{+}$ is the positive degree of any vertex of $H^s_{\mu}$.
    \end{proposition}
    \begin{proof}
        The proof is similar to proposition \ref{lm:Laplacian_coregular}.
    \end{proof}

 \begin{notation}
  \label{notation 3.1}
Consider a signed graph $G^s_\mu$ with $n$ vertices. Additionally, let there be signed graphs $H^s_{\mu_1}, H^s_{\mu_2},\dots, H^s_{\mu_n}$ with $t_1,t_2,\dots,t_n$ vertices, respectively. In order to get the desired expression for the signed characteristic polynomial, signed Laplacian Polynomial, and signed signless Laplacian polynomials, we introduce the following three polynomials involving n variables each 
    
 
    \begin{eqnarray*}
        g(\chi_{H^s_{\mu_1}}(\lambda)
        ,\dots, \chi_{H^s_{\mu_n}}(\lambda)
        ; G^s_\mu) & = & \det\left (\left[ 
            \begin{matrix}
        		\lambda-\chi_{H^s_{\mu_1}}(\lambda) & 0 & 0 \\
        		0& \ddots & 0 \\
        	    0 & 0 & \lambda-\chi_{H^s_{\mu_n}}(\lambda)
            \end{matrix} \right]-A(G^s_\mu )\right ), \\
    L({\chi}_{L({H^s_{\mu_1}})},\dots, {\chi}_{L({H^s_{\mu_n}})}; G^s_\mu) & = & \det\left (\left[ 
            \begin{matrix}
    		\alpha-{t_1}-{\chi }_{L(H^s_{\mu_1})}(\alpha) & 0 & 0 \\
    		0 & \ddots & 0 \\
    	    0 & 0 & \alpha-t_n-{\chi }_{L(H^s_{\mu_n})}(\alpha))
    	\end{matrix}
     \right] -L(G^s_\mu )\right),    \\
Q({\chi}_{Q(H^s_{\mu_1})}(\beta),\dots, {\chi}_{Q(H^s_{\mu_n})}(\beta); G^s_\mu) &=& \det \left (\left[ \begin{matrix}
    		\beta-t_1-{\chi}_{Q(H^s_{\mu_1})}(\beta) & 0 & 0 \\
    		0& \ddots & 0 \\
    	    0 & 0 & \beta-t_n-{\chi}_{Q(H^s_{\mu_n})}(\beta))
    		\end{matrix} \right]-Q(G^s_\mu)\right).		
     \end{eqnarray*}
     
       \end{notation}	
        
	\begin{remark}   {\label{Remark_01}}
    	Let $G^s_\mu $ be a signed graph of order $n$ and $H^s_{\mu_1}, H^s_{\mu_2},\dots, H^s_{\mu_n}$ be $n$ signed graphs. If $\chi_{H^s_{\mu_1}}(\lambda)=\chi_{H^s_{\mu_2}}(\lambda)\dots=\chi_{H^s_{\mu_n}}(\lambda)=\chi_{H^s_{\mu}}(\lambda)$ then
    	$$g(\chi_{H^s_{\mu_1}}(\lambda),\dots, \chi_{H^s_{\mu_n}}(\lambda); G^s_\mu )=f_{G^s_\mu }\left(\lambda-\chi_{H^s_{\mu }}(\lambda)\right)$$ 
where $f_{G^s_\mu}$ is the signed characteristic polynomial of ${G^s_\mu}$.
     
\end{remark}
	
    \begin{remark}{\label{remark 3.2}}
    	Let $G^s_\mu$ be a signed graph of order n and $H^s_{\mu_1}, H^s_{\mu_2},\dots, H^s_{\mu_n}$ be $n$ signed graphs each of them has order $m$. If ${\chi}_{L({H^s_{\mu_1}})}(\alpha)={\chi}_{L({H^s_{\mu_2}})}(\alpha)\dots={\chi}_{L({H^s_{\mu_n}})}(\alpha)={\chi}_{L({H^s_{\mu}})}(\alpha)$, then
        $$	L_{g}({\chi}_{L({H^s_{\mu_1}})}(\alpha),\dots {\chi}_{L({H^s_{\mu_n}})}(\alpha); G^s_\mu)=f_{L(G^s_\mu)}\left(\alpha-m-{\chi}_{L({H^s_{\mu}})}(\alpha)\right)$$
 where $f_{L(G^s_\mu)}\left(\alpha \right)$ is the signed Laplacian polynomial of ${G^s_\mu}$.       
    \end{remark}

    \begin{remark}  {\label{remark 3.3}}

    
        Let $G^s_\mu$ be a signed graph of order $n$ and $H^s_{\mu_1}, H^s_{\mu_2},\dots, H^s_{\mu_n}$ be $n$ signed graphs each of them has order $m$. If ${\chi}_{Q({H^s_{\mu_1}})}(\beta)={\chi }_{Q({H^s_{\mu_2}})}(\beta)\dots={\chi }_{Q({H^s_{\mu_n}})}(\beta )={\chi }_{Q({H^s_{\mu }})}(\beta )$ then
        $$Q_{g}({\chi }_{Q({H^s_{\mu_1}})}(\beta ),\dots, {\chi }_{Q({H^s_{\mu_n}})}(\beta ); G^s_\mu) = f_{Q(G^s_\mu)}\left( \beta -m-{\chi }_{Q({H^s_{\mu }})}(\beta )\right)$$
  where $f_{Q(G^s_\mu)}\left(\beta \right)$ is the signed signless Laplacian polynomial of ${G^s_\mu}$.      
    \end{remark}
    \subsection{Characteristic polynomials of \texorpdfstring{$G^s_\mu\circ \overset{n}{\underset{l=1}{ \Lambda}}H^s_{\mu_l}$}{}}
   We shall obtain the characteristic polynomial of the generalized corona product for signed graphs which will enable us to identify certain cospectral signed graphs. 
	
    \begin{theorem} 
    {\label{lm:char_polynomial}}
         Let $G^s_\mu$ be a signed graph with n vertices and $H^s_{\mu_1}, H^s_{\mu_2},\dots, H^s_{\mu_n}$ be n signed graphs (not necessarily non-isomorphic) with orders $t_1, \dots, t_n,$  respectively. Then
            
        	$$f_{G^s_\mu\circ \overset{n}{\underset{l=1}{ \Lambda}}H^s_{\mu_l}}(\lambda)= \left(\prod_{l=1}^{n} f_{H^s_{\mu_l}}(\lambda)\right) \cdot g(\chi_{H^s_{\mu_1}}(\lambda),\dots, \chi_{H^s_{\mu_n}}(\lambda); G^s_\mu).$$		
    \end{theorem}	
    \begin{proof} Consider Let $A$ and $B_l$ as the adjacency matrices of graphs $G^s_\mu$  and $H^s_{\mu_l}$, respectively, for $l=1,\dots,n.$ In light of the equations (\ref{eqn:pq})-(\ref{eqn:d}), it can be inferred that
$$f_{G^s_\mu \circ \overset{n}{\underset{l=1}{ \Lambda}}H^s_{\mu_l}}(\lambda)=\det \left[ \begin{matrix}
			\lambda I-A(G^s_\mu ) & PQ \\ \\
			(PQ)^T & \lambda I-D
		\end{matrix}
		\right],$$ By applying Lemma \ref{schur_complement}, we have
            \beano
                && f_{G^s_\mu\circ \overset{n}{\underset{l=1}{ \Lambda}}H^s_{\mu_l}}(\lambda) \\
                &=& \left(\prod_{l=1}^{n} f_{H^s_{\mu_l}}(\lambda)\right)\cdot \det \left( \lambda I-A(G^s_\mu)-PQ\left[ \begin{matrix}
                		\lambda I-{B_1} & 0 & 0 \\
                		0& \ddots & 0 \\                		
                	0 & 0 & \lambda I-{B_n}
                		\end{matrix} \right]^{-1}(PQ)^T\right ) \\
                &=& \left(\prod_{l=1}^{n} f_{H^s_{\mu_l}}(\lambda)\right)\cdot \det \left( \lambda I-A(G^s_\mu)-PQ\left[ \begin{matrix}
                		(\lambda I-{B_1})^{-1} & 0 & 0 \\
                		0& \ddots & 0 \\                		
                	0 & 0 & (\lambda I-{B_n})^{-1}
                		\end{matrix} \right]Q^TP^T\right )  \\
                &=& \left(\prod_{l=1}^{n} f_{H^s_{\mu_l}}(\lambda)\right) \cdot \det \left(\lambda I-A(G^s_\mu)-P
                		\bmatrix{
                	\mu_1[V_{\mu_1}]^T	(\lambda I-{B_1})^{-1}\mu_1[V_{\mu_1}] & 0 & 0 \\
                		0& \ddots & 0 \\
                	0 & 0 &\mu_n[V_{\mu_n}]^T( \lambda I-{B_n})^{-1} \mu_n[V_{\mu_n}]
                		} P^T\right) \\
                &=& \left(\prod_{l=1}^{n} f_{H^s_{\mu_l}}(\lambda)\right) \cdot \det \left (  \lambda I-A(G^s_\mu)-\left[ \begin{matrix}
                		\chi_{H^s_{\mu_1}}(\lambda)& 0 & 0 \\
                		0& \ddots & 0 \\
                	0 & 0 &\chi_{H^s_{\mu_n}}(\lambda)
                		\end{matrix} \right]\right) \\
                			&=& \left(\prod_{l=1}^{n} f_{H^s_{\mu_l}}(\lambda)\right) \cdot \det \left( \bmatrix{
                		\lambda -\chi_{H^s_{\mu_1}}(\lambda)& 0 & 0 \\
                		0& \ddots & 0 \\
                	0 & 0 &\lambda-\chi_{H^s_{\mu_n}}(\lambda)}-A(G^s_\mu)\right)  \\
                		&=& \left(\prod_{l=1}^{n} f_{H^s_{\mu_l}}(\lambda)\right) \cdot g(\chi_{H^s_{\mu_1}}(\lambda),\dots, \chi_{H^s_{\mu_n}}(\lambda); G^s_\mu)
            \eeano 
   This validates the proof.     \end{proof} 
The following corollary is a direct consequence of the theorem \ref{lm:char_polynomial}.
    \begin{corollary}
       
If $G^s_{\mu_1}$ and $H^s_{{\mu}^{\prime}}$ are two signed graphs having orders $n$ and $m$ respectively then the following holds:
        $$f_{G^s_{\mu}\circ H^s_{\mu^{\prime}}}(\lambda)=\left(f_{\chi_{H^s_{\mu_2}}}(\lambda)\right)^n \cdot f_{G^s_{\mu_1}}\left(\lambda-\chi_{H^s_{\mu_2}}(\lambda)\right).$$    
    \end{corollary}

    \begin{proof}
        Let $H^s_{\mu_1} \cong \ldots \cong H^s_{\mu_n} \cong H^s_{{\mu }^{\prime}}$. Then utilizing theorem \ref{lm:char_polynomial} and proposition \ref{lm:isomorphic_equalcoronal}, we deduce
         \[
         f_{G^s_{\mu}\circ H^s_{\mu^{\prime}}}(\lambda)=\left(f_{\chi_{H^s_{\mu_2}}}(\lambda)\right)^n \cdot f_{G^s_{\mu_1}}\left(\lambda-\chi_{H^s_{\mu_2}}(\lambda)\right) \qedhere
         \]
         \end{proof}
         
    \begin{corollary} {\label{characteristicpolynomial_knetregularcase}}
        Suppose $G^s_{\mu}$ is a signed graph of order n. Let $H^s_{\mu_1}, H^s_{\mu_2},\dots, H^s_{\mu_n}$ be n signed graphs such that $\chi_{H^s_{\mu_1}}(\lambda)=\dots=\chi_{H^s_{\mu_n}}(\lambda)=\chi_{H^s_{\mu^\prime}}(\lambda)$. Then
        $$f_{G^s_{\mu}\circ \overset{n}{\underset{l=1}{ \Lambda}}H^s_{\mu_l}}(\lambda)=\left(\prod_{l=1}^{n} f_{H^s_{\mu_l}}(\lambda)\right) \cdot f_{G^s_{\mu}}\left(\lambda-\chi_{H^s_{\mu^\prime}}(\lambda)\right)$$ 
        
    \end{corollary}
    \begin{proof}
        It is a straight forward implication of theorem \ref{lm:char_polynomial} and remark \ref{Remark_01}.
    \end{proof}

    \begin{corollary}
        Suppose $G^s_{\mu}$ is a signed graph of order $n$. Let $H^s_{\mu_1}, H^s_{\mu_2},\dots, H^s_{\mu_n}$ be a collection of $k$-net regular signed graphs and each of them consists of $m$ vertices. Then $$f_{G^s_{\mu}\circ \overset{n}{\underset{l=1}{ \Lambda}}H^s_{\mu_l}}(\lambda)=\left(\prod_{l=1}^{n} f_{H^s_{\mu_l}}(\lambda)\right).f_{G^s_{\mu}}\left(\lambda-\frac{m}{\lambda-k}\right)$$
    
    \end{corollary}
\begin{proof}
It directly follows from proposition \ref{formula_signedcoronal} and corollary \ref{characteristicpolynomial_knetregularcase}.
\end{proof}

    \begin{corollary}
        Let $G^s_{\mu^\prime}$ and $G^s_{\mu^{\prime\prime}}$ be two co-spectral signed graphs of order $n$ each and $H^s_{\mu_1}, H^s_{\mu_2},\dots, H^s_{\mu_{n}}$ be $n$ signed graphs. If $\chi_{H^s_{\mu_1}}(\lambda)=\chi_{H^s_{\mu_2}}(\lambda)\dots=\chi_{H^s_{\mu_n}}(\lambda)=\chi_{H^s_{\mu}}(\lambda)$  then $G^s_{\mu^\prime}\circ \overset{n}{\underset{l=1}{ \Lambda}}H^s_{\mu_l}(\lambda)$ and $G^s_{\mu^{\prime\prime}}\circ \overset{n}{\underset{l=1}{ \Lambda}}H^s_{\mu_l}(\lambda)$ are co-spectral. 
    \end{corollary}

    \begin{proof}
    From theorem \ref{lm:char_polynomial} we have
    \begin{equation*}
    \begin{split}
     f_{G^s_{\mu^\prime}\circ \overset{n}{\underset{l=1}{ \Lambda}}H^s_{\mu_l}}(\lambda) \\
    &= \left( \prod_{l=1}^{n} f_{H^s_{\mu_l}}(\lambda)\right) \det \left( \bmatrix{
    		\lambda-\chi_{H^s_{\mu}}(\lambda) & 0 & 0 \\
    		0& \ddots & 0 \\
    	0 & 0 & \lambda-\chi_{H^s_{\mu}}(\lambda)}-A(G^s_{\mu^\prime})\right) \\
    &= \left(\prod_{l=1}^{n} f_{H^s_{\mu_l}}(\lambda)\right) \cdot f_{G^s_{\mu^\prime}}\left(\lambda-\chi_{H^s_{\mu}}(\lambda) \right) \\
    &= \left(\prod_{l=1}^{n} f_{H^s_{\mu_l}}(\lambda)\right) \cdot f_{G^s_{\mu^{\prime\prime}}}\left(\lambda-\chi_{H^s_{\mu}}(\lambda) \right) \\
    &= f_{G^s_{\mu^{\prime\prime}} \circ \overset{n}{\underset{l=1}{ \Lambda}}H^s_{\mu_l}}(\lambda).  \qedhere
    \end{split}
    \end{equation*} 
    \end{proof}
The corollary below is readily provable using Theorem \ref{lm:char_polynomial}.
    \begin{corollary}
    Let $G^s_{\mu}$ be a signed graph of order $n$ and $H^s_{\mu_1}, H^s_{\mu_2},\dots, H^s_{\mu_{2n}}$ be a family of cospectral signed graphs such that $\chi_{H^s_{\mu_1}}(\lambda)=\chi_{H^s_{\mu_2}}(\lambda)\dots=\chi_{H^s_{\mu_{2n}}}(\lambda)$  then $G^s_\mu \circ \overset{n}{\underset{l=1}{ \Lambda}}H^s_{\mu_l}(\lambda)$ and $G^s_\mu \circ \overset{2n}{\underset{l=n+1}{ \Lambda}}H^s_{\mu_l}(\lambda)$ are cospectral. 
    
    \end{corollary}

    \begin{definition} \cite{pirzada2008anote}
     A signed Graph G is called a signed bipartite graph if its set of vertices can be divided into two disjoint sets, M and N, such that each edge connects a vertex from M to a vertex from N. If the signed graph is bipartite, we will denote it by $G(M, N)$ where $M$ and $N$ are the disjoint partitions of the vertex set of G. 
    
    \end{definition}
\begin{notation}  {\label{notation_4.1}}
    Suppose $G^s_\mu(M,N)$ is a signed bipartite graph of order $n$, where $|M|=h$ and $|N|=n-h$. Also, assume that $W$ is a matrix of order $h \times (n-h)$ such that
            $$ A(G^s_\mu)= \left[ \begin{matrix}
        			 0 & W \\ \\
        			W^T & 0
        			\end{matrix}
        			\right]$$
                    Therefore, 
                    $$f_{G^s_\mu}(\lambda)= \det(\lam I- A(G))= \det \left[ \begin{matrix}
        			\lambda I_h & -W \\ \\
        			-W^T & \lambda I_{n-h}
        			\end{matrix}
        			\right]$$	
        and with two methods of Schur complement \ref{schur_complement}, we have:
        \begin{equation*} \label{eq2}
        f_{G^s_\mu}(\lambda) =
        \begin{cases}
            \lambda^{n-2h} \det (\lambda^2 I_h-WW^T)=\lambda^{n-2h}g_{G^s_\mu}(\lambda). \\ \\
            \lambda^{2h-n} \det (\lambda^2 I_{n-h}-W^TW)=\lambda^{2h-n}q_{G^s_\mu}(\lambda).
        \end{cases}
        \end{equation*}	
        If $n=2h$ then it is clear that $f_{G^s_\mu}(\lambda)=g_{G^s_\mu}(\lambda)=q_{G^s_\mu}(\lambda)$.
        \end{notation}
    \begin{theorem}   {\label{th:square_root}}
    Let $G^s_\mu (M,N)$ be a signed bipartite graph of order $n$, where $|M|=i$ and $|N|=n-i$. Let $H^s_{\mu_1} \cong \ldots \cong H^s_{\mu_i} \cong Z^s_{{\mu }^{'}}$ 
    and $H^s_{\mu_{i+1}} \cong \ldots \cong H^s_{\mu_n} \cong Z^s_{{\mu }^{''}}$. Then the following hold: \newline
    \begin{equation*}
    f_{G^s_{\mu}\circ \overset{n}{\underset{l=1}{ \Lambda}}H^s_{\mu_l}}(\lambda) =
        \begin{cases}
            {(f_{Z^s_{{\mu}^\prime}}(\lambda))}^{i}.{(f_{Z^s_{\mu^{\prime\prime}}}(\lambda))}^{n-i}{(\lambda-\chi_{Z^s_{{\mu}^{\prime\prime}}}(\lambda))}^{n-2i}g_G\left(\sqrt{(\lambda-\chi_{Z^s_{\mu^\prime}(\lambda)})(\lambda-\chi_{Z^s_{\mu^{\prime\prime}}(\lambda)})}\right) & if ~~ n \ge 2i  \\ \\
            {(f_{Z^s_{{\mu}^\prime}}(\lambda ))}^{i}.{(f_{Z^s_{\mu^{\prime\prime}}}(\lambda))}^{n-i}{(\lambda-\chi_{Z^s_{{\mu}^{\prime}}}(\lambda))}^{2i-n}q_{G}\left(\sqrt{(\lambda-\chi_{Z^s_{\mu^\prime}(\lambda)})(\lambda-\chi_{Z^s_{\mu^{\prime\prime}}(\lambda)})}\right) & if ~~ n \le 2i
        \end{cases}
    \end{equation*}
    \end{theorem}
\begin{proof}
    Based on the identical argument presented in the proof of Theorem \ref{lm:char_polynomial}, it follows that.
    \begin{eqnarray}
f_{G^s_\mu\circ \overset{n}{\underset{l=1}{ \Lambda}}H^s_{\mu_l}}(\lambda) 
	&=&  \left(\prod_{l=1}^{n} f_{{H^s_{\mu_l}}}(\lambda)\right)\cdot \det \left( \bmatrix{
		{\lambda-\chi}_{{H^s_{\mu_1}}}(\lambda)& 0 & 0 \\
		0& \ddots & 0 \\
		0 & 0 &{\lambda-\chi}_{{H^s_{\mu_n}}}(\lambda)}-A(G)\right)  \nonumber \\   
			&=&  \left(f_{Z^s_{{\mu}^\prime}}(\lambda)\right)^i \cdot \left(f_{Z^s_{{\mu}^{\prime\prime}}}(\lambda)\right)^{n-i} \cdot \det \left( \bmatrix{
		{\lambda-\chi}_{{Z^s_{\mu^\prime}}}(\lambda)& 0 \\ \\
		0 & {\lambda-\chi}_{{Z^s_{\mu^{\prime\prime}}}}(\lambda)}-A(G)\right).\nonumber
		\end{eqnarray}
			
		Since $G^s_\mu$ is bipartite, there is a matrix $W$ of order $i \times (n-i)$ such that 
	$$f_{G^s_\mu\circ \overset{n}{\underset{l=1}{ \Lambda}}H^s_{\mu_l}}(\lambda) 
		=  \left(f_{Z^s_{{\mu}^\prime}}(\lambda)\right)^i \cdot (f_{Z^s_{{\mu}^{\prime\prime}}}(\lambda))^{n-i} \cdot \det \bmatrix{
		{\lambda-\chi}_{{Z^s_{\mu^\prime}}}(\lambda)& -W \\ \\
		-W^{T} & {\lambda-\chi}_{{Z^s_{\mu^{\prime\prime}}}}(\lambda)}.$$

By utilizing Lemma \ref{schur_complement} and Notation 4, we can derive that.
 \newline
    \begin{equation*}
    f_{G^s_{\mu}\circ \overset{n}{\underset{l=1}{ \Lambda}}H^s_{\mu_l}}(\lambda) =
        \begin{cases}
            {(f_{Z^s_{{\mu}^\prime}}(\lambda))}^{i}.{(f_{Z^s_{\mu^{\prime\prime}}}(\lambda))}^{n-i}{(\lambda-\chi_{Z^s_{{\mu}^{\prime\prime}}}(\lambda))}^{n-2i}g_G\left(\sqrt{(\lambda-\chi_{Z^s_{\mu^\prime}(\lambda)})(\lambda-\chi_{Z^s_{\mu^{\prime\prime}}(\lambda)})}\right) & if ~~  n \ge 2i  \\ \\
            {(f_{Z^s_{{\mu}^\prime}}(\lambda ))}^{i}.{(f_{Z^s_{\mu^{\prime\prime}}}(\lambda))}^{n-i}{(\lambda-\chi_{Z^s_{{\mu}^{\prime}}}(\lambda))}^{2i-n}q_G\left(\sqrt{(\lambda-\chi_{Z^s_{\mu^\prime}(\lambda)})(\lambda-\chi_{Z^s_{\mu^{\prime\prime}}(\lambda)})}\right) & if ~~ n \le 2i 
        \end{cases} \qedhere
    \end{equation*}
		\end{proof}
To supply the subsequent corollaries, we shall now introduce the concept of the complement of a connected marked signed graph.
  \begin{definition}{\label{definition 4.2}}
      Let $G^s_{\mu_1}=(V_{\mu_1},E^s_1,\sigma_1)$ be a connected marked signed graph. The graph $G^s_{\mu_2}=(V_{\mu_2},E^s_2,\sigma_2)$ is called complement of $G^s_{\mu_1}$ if $V(G^s_{\mu_1})=V(G^s_{\mu_2})$
  with $\mu_1=\mu_2=\mu(say)$ and $E(G^s_{\mu_2})=\{uv: u, v \in V(G^s_{\mu_1}) \mbox{and u and v are not adjacent in}\,G^s_{\mu_1}\}$ such that $\sigma_2(uv)=\mu(u)\mu(v)$.
  
  \end{definition}

 It is easy to see that ${\chi_{\bar K^s_n}(\lambda)}=\frac{n}{\lambda}$, where $\bar K^s_n$ is the complement of the complete signed graph $K^s_n$ on n vertices. Therefore, the subsequent corollaries can be easily concluded from Theorem \ref{th:square_root}.
	
	\begin{corollary}
	Let $G^s_\mu(M,N)$ be a signed bipartite graph of order n, where $|M|=i$ and $|N|=n-i$. Let $H^s_{\mu_1} \cong \ldots \cong H^s_{\mu_i} \cong \bar K^s_{{m}}$ and $H^s_{\mu_{i+1}} \cong \ldots \cong H^s_{\mu_n} \cong \bar K^s_{{t}}$. Then we have
$$f_{G^s_{\mu}\circ \overset{n}{\underset{l=1}{ \Lambda}}H^s_{\mu_l}}(\lambda) = \left\{ \begin{array}{rcl}
\lambda^{mi+(n-i)t}.(\lambda-\frac{t}{\lambda})^{n-2i}g_G\left(\sqrt{(\lambda-\frac{m}{\lambda})(\lambda-    \frac{t}{\lambda})}\right) & if & n \ge 2i  \\ \\
\lambda^{mi+(n-i)t}.(\lambda-\frac{m}{\lambda})^{2i-n}q_G\left(\sqrt{(\lambda-\frac{m}{\lambda})(\lambda-    \frac{t}{\lambda})}\right) & if & n \le 2i 
\end{array}\right.$$
	\end{corollary}
	\begin{corollary}
	Let $G^s_\mu(M,N)$ be a signed bipartite graph of order n, where $|M|=i$ and $|N|=n-i$. Let $H^s_{\mu_1} \cong \ldots \cong H^s_{\mu_i} \cong \bar K^s_{{m}}$ and $H^s_{\mu_{i+1}} \cong \ldots \cong H^s_{\mu_n} \cong \phi$. Then we have
\begin{equation*}
f_{G^s_{\mu}\circ \overset{n}{\underset{l=1}{ \Lambda}}H^s_{\mu_l}}(\lambda) =
    \begin{cases}
        \lambda^{mi+n-2i}g_G\left(\sqrt{(\lambda-\frac{m}{\lambda})\lambda}\right) & if ~~ n \ge 2i  \\ \\
        \lambda^{mi}.(\lambda-\frac{m}{\lambda})^{2i-n}q_G\left(\sqrt{(\lambda-\frac{m}{\lambda})\lambda}\right) & if ~~ n \le 2i.
    \end{cases}
\end{equation*}
	\end{corollary}

    \subsection{Laplacian polynomials of \texorpdfstring{$G^s_\mu \circ \overset{n}{\underset{l=1}{ \Lambda}}H^s_{\mu_l}$}{}}
    In the proof of the subsequent theorem, we utilize a notation involving a set of graphs denoted as $H^s_{\mu_1}, H^s_{\mu_2},\dots, H^s_{\mu_{n}}$, with each individual graph having orders $t_1,t_2,\dots,t_n$ respectively. The matrix V is defined to be diagonal with its diagonal entries being $t_1$ through $t_n$. $\Delta (G^s_\mu)$ denotes the diagonal matrix with its entries as the degree sequence of $G^s_\mu$. Similarly, $\Delta(H^s_{\mu_l})$ denotes the diagonal matrix with its entries as the degree sequence of $H^s_{\mu_l}$ $(1 \leq l \leq n)$. $\Delta$ is the diagonal matrix with its entries being $\Delta(H^s_{\mu_1})$ through $\Delta(H^s_{\mu_n})$.
       
\begin{theorem}{\label{th:signedLaplacianPolynomial}}
   Suppose $G^s_\mu$ is a signed graph comprising n vertices and let $H^s_{\mu_1}, H^s_{\mu_2},\dots, H^s_{\mu_{n}}$ represent $n$ signed graphs of orders $t_1,t_2,\dots,t_n$ respectively, not necessarily non-isomorphic. Then

     $$f_{L({G^s_\mu\circ \overset{n}{\underset{l=1}{ \Lambda}}H^s_{\mu_l}})}(\alpha)= \left(\prod_{l=1}^{n} f_{L({H^s_{\mu_l}})}(\alpha-1)\right) \cdot L_g({\chi}_{L({H^s_{\mu_1}})}(\alpha),\dots {\chi}_{L({H^s_{\mu_n}})}(\alpha); G^s_\mu).$$
    
    \end{theorem}
\begin{proof}
Consider $A(G^s_\mu)$ as the adjacency matrix of graph G and $B_l$ as the adjacency matrix of graph $H^s_{\mu_l}$ for each $l=1,\dots ,n.$ Utilizing the equations (\ref{eqn:pq})-(\ref{eqn:d}), it can be deduced that

\begin{eqnarray*}
&& f_{L({G^s_\mu \circ \overset{n}{\underset{l=1}{ \Lambda}}H^s_{\mu_l}})}(\alpha)\\ \\
&=& \det \bmatrix{
		\alpha I_{n}-V-\Delta (G^s_\mu)+A(G^s_\mu) & PQ \\ \\
		(PQ)^T&  \alpha I_{{(\sum_{l=1}^nt_l) \times \sum_{l=1}^nt_l)}}-\Delta-I_{{(\sum_{l=1}^nt_l) \times \sum_{l=1}^nt_l)}}+D} \\ \\
&=& \det \bmatrix{
		\alpha I-V-\Delta (G^s_\mu)+A(G^s_\mu) & PQ \\ \\
		(PQ)^T& 	
			\bmatrix{
		(\alpha-1) I-\Delta(H^s_{\mu_1})+B_1 & 0&0 \\
		0& \ddots &0 \\
		0&0& (\alpha-1) I-\Delta(H^s_{\mu_n})+B_n
		} \\
		}\\ \\
&=& \det \bmatrix{
		\alpha I-V-L(G^s_\mu) & PQ \\ \\
		(PQ)^T& 	
			\bmatrix{
		(\alpha-1) I-L(H^s_{\mu_1})& 0&0 \\
		0& \ddots &0 \\
		0&0& (\alpha-1) I-L(H^s_{\mu_n})
		}\\
		}\\
  \end{eqnarray*}
  By applying Lemma \ref{schur_complement}, we arrive at
  \begin{eqnarray*}
&=& \left(\prod_{l=1}^{n} f_{L({H^s_{\mu_l}})}(\alpha-1)\right) \\
&& \cdot \det \left(\alpha I-V-L(G^s_\mu) -(PQ)
		\bmatrix{
	(\alpha-1)I-L(H^s_{\mu_1}) & 0 & 0 \\
		0& \ddots & 0 \\
	0 & 0 & (\alpha-1)I-L(H^s_{\mu_n})}^{-1} (PQ)^T\right)
 \end{eqnarray*}
 \begin{eqnarray*}
&=& \left(\prod_{l=1}^{n} f_{L({H^s_{\mu_l}})}(\alpha-1)\right) \\
&& \cdot \det  \left( \alpha I-V-L(G^s_\mu)-PQ
		\bmatrix{
	((\alpha-1)I-L(H^s_{\mu_1}))^{-1} & 0 & 0 \\
		0& \ddots & 0 \\
	0 & 0 & ((\alpha-1)I-L(H^s_{\mu_n}))^{-1}} (PQ)^T\right)  \\ \\
&=&   \left(\prod_{l=1}^{n} f_{L({H^s_{\mu_l}})}(\alpha-1)\right)  \cdot \det   \left( \alpha I-V-L(G^s_\mu)- \right. \\ \\
 && \left. P 
    	\bmatrix{\mu_1[V_1]^TI_{t_1}((\alpha-1)I-L(H^s_{\mu_1}))^{-1}\mu_1[V_1]I_{t_1} & 0 & 0 \\
		0& \ddots & 0 \\
	0 & 0 &\mu_1[V_n]^T I_{t_n} ( (\alpha-1)I-L(H^s_{\mu_n}))^{-1}\mu_1[V_n]}I_{t_n}P^T \right) \\ \\
&=&  \left(\prod_{l=1}^{n} f_{L({H^s_{\mu_l}})}(\alpha-1)\right)  \cdot \det \left( \alpha I-V-L(G^s_\mu)-\bmatrix{
		{\chi}_{L({H^s_{\mu_1}})}(\alpha)& 0 & 0 \\
		0& \ddots & 0 \\
	0 & 0 &{\chi}_{L({H^s_{\mu_n}})}(\alpha)} \right)  \\ \\
&=&  \left(\prod_{l=1}^{n} f_{L({H^s_{\mu_l}})}(\alpha-1)\right) \cdot \det \left( \bmatrix{
		\alpha -t_1 -{\chi}_{L({H^s_{\mu_1}})}(\alpha)& 0 & 0 \\
		0& \ddots & 0 \\
	0 & 0 &\alpha -t_n-{\chi}_{L({H^s_{\mu_n}})}(\alpha)} -L(G^s_\mu)\right)
 \end{eqnarray*}
  Under the notation given in Notation \ref{notation 3.1}, the result is that
  \begin{equation*}
f_{L({G^s_\mu\circ \overset{n}{\underset{l=1}{ \Lambda}}H^s_{\mu_l}})}(\alpha)= \left(\prod_{l=1}^{n} f_{L({H^s_{\mu_l}})}(\alpha-1)\right) \cdot L_g({\chi}_{L({H^s_{\mu_1}})}(\alpha),\dots, {\chi}_{L({H^s_{\mu_n}})}(\alpha); G^s_\mu) \qedhere
\end{equation*} 
\end{proof}    
\begin{corollary}
 Suppose $G^s_{\mu}$  and $H^s_{{\mu}^{\prime}}$ represent two signed graphs with n and m vertices, respectively. Then

$$f_{L({G^s_{\mu}\circ H^s_{\mu^{\prime}})}}(\alpha)=\left(f_{L(H^s_{{\mu}^{\prime}})}(\alpha-1)\right)^n \cdot f_{L(G^s_{\mu})}\left(\alpha-m- {\chi}_{L({H^s_{{\mu}^{\prime}}})}(\alpha)\right)$$
\end{corollary}

\begin{proof}
Assuming $H^s_{\mu_1} \cong \ldots \cong H^s_{\mu_n} \cong {H}^s_{\mu^\prime}$, the application of theorem \ref{th:signedLaplacianPolynomial} along with the consideration of remark \ref{remark 3.2} leads to the desired result.
\end{proof}
\begin{corollary}
     Let $G^s_{\mu}$ be a signed graph of order $n$ and $H_1,\dots,  H_n$ be $n$ signed graphs, each of them has order $m$ such that ${\chi}_{L({H^s_{\mu_1}})}(\alpha)={\chi}_{L({H^s_{\mu_2}})}(\alpha)\dots={\chi}_{L({H^s_{\mu_n}})}(\alpha)={\chi}_{L({H^s_{{\mu}^\prime}})}(\alpha)$. Then
       $$f_{L({G^s_\mu\circ \overset{n}{\underset{l=1}{ \Lambda}}H^s_{\mu_l}})}(\alpha)=\left(\prod_{l=1}^{n} f_{L({H^s_{\mu_l}})}(\alpha-1)\right).f_{L(G^s_{\mu})}\left(\alpha-m-\\ {\chi}_{L({H^s_{{\mu}^\prime}})}(\alpha)\right).$$
     \end{corollary}
     \begin{proof}
         It is clear from Theorem \ref{th:signedLaplacianPolynomial} and Remark \ref{remark 3.2}.
     \end{proof}


    \begin{corollary}
       Suppose $G^s_\mu$ represents a signed graph consisting of $n$ vertices. Let $H^s_{\mu_1}, H^s_{\mu_2},\dots, H^s_{\mu_{n}}$ be $n$ co-regular signed graphs each having an order of $m$ and a common co-regularity pair $(r, k)$. Then

        $$f_{L({G^s_\mu \circ \overset{n}{\underset{l=1}{ \Lambda}}H^s_{\mu_l}})}(\alpha)= \left(\prod_{l=1}^{n} f_{L({H^s_{\mu_l}})}(\alpha-1)\right).f_{L(G^s_\mu)}\left(\alpha-m-\frac{m}{\alpha-1-2d^{-}}\right)$$
        where $d^{-}$ is the negative degree of any vertex of $H^s_{{\mu}_{l}}$ $(1 \leq l \leq n)$.
    \end{corollary}
    \begin{proof}
        It is clear from theorem \ref{th:signedLaplacianPolynomial} and proposition \ref{lm:Laplacian_coregular}.
    \end{proof}
    The implications of Theorem \ref{th:signedLaplacianPolynomial} directly validate the assertions made in the subsequent Corollaries.
    \begin{corollary}
     Suppose $G^s_{{\mu}^\prime}$ and $G^s_{{\mu}^{\prime\prime}}$ represent two $L$-cospectral signed graphs of order $n$ and let $H^s_{\mu_1}, H^s_{\mu_2},\dots, H^s_{\mu_n}$ denote a collection of $n$ signed graphs each with order $m$. Then $G^s_{{\mu }^\prime}\circ \overset{n}{\underset{l=1}{\Lambda}}H^s_{\mu_l}$  and $G^s_{\mu^{\prime\prime}}\circ \overset{n}{\underset{l=1}{ \Lambda}}H^s_{\mu_l}$ are L-cospectral. 
\end{corollary}
    
     \begin{corollary}
Suppose $G^s_\mu$ is a signed graphs of size $n$ and let $H^s_{\mu_1}, H^s_{\mu_2},\dots, H^s_{\mu_{2n}}$ be a collection of $L$-cospectral signed graphs each with order $m$. Then $G^s_\mu \circ \overset{n}{\underset{l=1}{\Lambda}} H^s_{\mu_l}$  
        and $G^s_\mu \circ \overset{2n}{\underset{l=n+1}{ \Lambda}}H^s_{\mu_l}$ are $L$-cospectral. 
        \end{corollary}
 
 \begin{theorem}{\label{Theorem 5.2}}
     Let $G^s_\mu$ be a signed bipartite graph with $|V(G)|=n=2k$ where $k$ denotes the size of each part. Let the underlying graph of G be r-regular. Suppose $H^s_{\mu_1} \cong \ldots \cong H^s_{\mu_k} \cong Z^s_{{\mu}^1}$ and $H^s_{\mu_{k+1}} \cong \ldots \cong H^s_{\mu_n} \cong Z^s_{{\mu}^{2}}$. If $|V(Z^s_{\mu^1})|=m$ and $|V(Z^s_{\mu^{2}})|=s$ then
     $$f_{L({G^s_\mu \circ \overset{n}{\underset{l=1}{ \Lambda}}H^s_{\mu_l}})}(\alpha)=\prod_{j=1}^{2} ( f_{L({Z^s_{\mu^{j}}})}(\alpha-1))^k.f_{G^s_\mu}\left(\sqrt{\left(\alpha-m-r-{\chi}_{L({Z^s_{\mu^1}})}(\alpha)\right)\left(\alpha-s-r-{\chi}_{L({Z^s_{\mu^2}})}(\alpha)\right)}\right)$$
     \end{theorem}
    \begin{proof}
       Through a parallel line of reasoning as demonstrated in the theorem \ref{th:signedLaplacianPolynomial} proof we have\\
    \beano
        \small{f_{L({G^s_\mu \circ \overset{n}{\underset{l=1}{ \Lambda}}H^s_{\mu_l}})}}(\alpha)  =  \prod_{j=1}^{2} {( f_{L({Z^s_{\mu^{j}}})} (\alpha-1))}^k.det \left( \left[ \begin{matrix}
        		(\alpha-m-r-{\chi}_{L({Z^s_{\mu^1}})}(\alpha))I_k & 0 \\ \\
        		0 &  (\alpha-s-r-{\chi}_{L({Z^s_{\mu^2}})}(\alpha))I_k \\
        \end{matrix} \right] + A(G^s_\mu) \right)
        \eeano 
Due to the bipartite nature of $G^s_{\mu}$, the existence of a $k$-order matrix $W$ is guaranteed such that \\
    \begin{equation*}
f_{L({G^s_\mu \circ \overset{n}{\underset{l=1}{ \Lambda}}H^s_{\mu_l}})}(\alpha)  = \prod_{j=1}^{2} ( f_{L({Z^s_{\mu^{j}}})}(\alpha-1))^k.det \left(\left[ \begin{matrix}
		(\alpha-m-r-{\chi}_{L({Z^s_{\mu^1}})}(\alpha))I_k & W \\ \\
		W^T &  (\alpha-s-r-{\chi}_{L({Z^s_{\mu^2}})}(\alpha))I_k \\
			\end{matrix} \right]\right) 
			\end{equation*}
Using both lemma \ref{schur_complement} and notation \ref{notation_4.1} we obtain that
\begin{equation*}
     f_{L({G^s_\mu \circ \overset{n}{\underset{l=1}{ \Lambda}}H^s_{\mu_l}})}(\alpha)=\prod_{j=1}^{2} ( f_{L({Z^s_{\mu^{j}}})}(\alpha-1))^k.f_{G^s_\mu}\left(\sqrt{\left(\alpha-m-r-{\chi}_{L({Z^s_{\mu^1}})}(\alpha)\right)\left(\alpha-s-r-{\chi}_{L({Z^s_{\mu^2}})}(\alpha)\right)}\right)   \qedhere
     \end{equation*}
   
\end{proof} 		
		
	\begin{corollary}  {\label{corollary 5.2.1}}
     Let $G^s_\mu$ be a signed co-regular bipartite graph of co-regularity pair $(r,k)$ and with $|V(G)|=n=2k$ where $k$ denotes the
size of each part. Suppose $H^s_{\mu_1} \cong \ldots \cong H^s_{\mu_k} \cong Z^s_{{\mu}^1}$ and $H^s_{\mu_{k+1}} \cong \ldots \cong H^s_{\mu_n} \cong Z^s_{{\mu}^{2}}$. If $|V(Z^s_{\mu^1})|=m$ and $|V(Z^s_{\mu^{2}})|=s$ then
     $$f_{L({G^s_\mu \circ \overset{n}{\underset{l=1}{ \Lambda}}H^s_{\mu_l}})}(\alpha)=\prod_{j=1}^{2} ( f_{L({Z^s_{\mu^{j}}})}(\alpha-1))^k.f_{G^s_\mu}\left(\sqrt{\left(\alpha-m-r-\frac{m}{\alpha-1-2{d^-_1}}\right)\left(\alpha-s-r-\frac{s}{\alpha-1-2{d^-_2}}\right)}\right)$$
     where ${d}^-_1$ and ${d}^-_2$ are the negative degree of any vertex of $Z^s_{\mu^1}$ and $Z^s_{\mu^2}$ respectively. 
     \end{corollary}	
     \begin{proof}
        It is straightforward from theorem \ref{Theorem 5.2} and proposition \ref{lm:Laplacian_coregular}. 
     \end{proof}

 \begin{corollary}  {\label{corollary 5.2.2}}
        Let $G^s_\mu$ be a signed co-regular bipartite graph of co-regularity pair $(r, k)$ and with $|V(G)|=n=2k$ where $k$ denotes the
size of each part. Suppose $H^s_{\mu_1} \cong \ldots \cong H^s_{\mu_k} \cong \bar{K}_m$ and $H^s_{\mu_{k+1}} \cong \ldots \cong H^s_{\mu_n} \cong \bar{K}_s$. Then the following holds:
        $$f_{L({G^s_\mu \circ \overset{n}{\underset{l=1}{ \Lambda}}H^s_{\mu_l}})}(\alpha) = {(\alpha -1)}^{k(m+s)}.f_{G^s_\mu }\left( \sqrt{\left( \alpha-m-r-\frac{m}{\alpha-1}\right)\left(\alpha-s-r-\frac{s}{\alpha-1}\right)}\right).$$
    \end{corollary}
    \begin{proof}
It is a direct consequence of corollary \ref{corollary 5.2.1} and definition \ref{definition 4.2}.
    \end{proof}

    \begin{corollary}
        Let $G^s_\mu$ be a signed co-regular bipartite graph of co-regularity pair $(r, k)$ and with $|V(G)|=n=2k$ where k is the size of each part. Let $H^s_{\mu_1} \cong \ldots \cong H^s_{\mu_k} \cong \bar{K}_m$ and $H^s_{\mu_{k+1}} \cong \ldots \cong H^s_{\mu_n} \cong \phi$. Then the following holds:
        $$f_{L({G^s_\mu \circ \overset{n}{\underset{l=1}{ \Lambda}}H^s_{\mu_l}})}(\alpha)=(\alpha-1))^{km}.f_{G^s_\mu}\left(\sqrt{\left(\alpha-m-r-\frac{m}{\alpha-1}\right)\left(\alpha-r\right)}\right).$$
    \end{corollary}	
    \begin{proof}
    It can be easily obtained by corollary \ref{corollary 5.2.2} for $s=0$.    
    \end{proof}
     \subsection{Signless Laplacian polynomials of \texorpdfstring{$G^s_\mu \circ \overset{n}{\underset{l=1}{ \Lambda}}H^s_{\mu_l}$}{}}
\begin{theorem}   {\label{signlesslaplacianpolynomial}}
   Suppose $G^s_\mu$ is a signed graph comprising n vertices and let $H^s_{\mu_1}, H^s_{\mu_2},\dots, H^s_{\mu_{n}}$ represent $n$ signed graphs of orders $t_1,t_2,\dots,t_n$ respectively, not necessarily non-isomorphic. Then
     $$f_{Q({G^s_\mu\circ \overset{n}{\underset{l=1}{ \Lambda}}H^s_{\mu_l}})}(\beta)=(\prod_{l=1}^{n} f_{Q({H^s_{\mu_l}})}(\beta-1)).Q_g({\chi}_{Q({H^s_{\mu_1}})}(\beta),\dots, {\chi}_{Q({H^s_{\mu_n}})}(\beta); G^s_\mu).$$
    \end{theorem}
\begin{proof}
The proof is similar to Theorem \ref{th:signedLaplacianPolynomial}.
\end{proof}    
\begin{corollary}
 Suppose $G^s_{\mu}$  and $H^s_{{\mu}^{\prime}}$ represent two signed graphs with n and m vertices, respectively. Then

$$f_{Q({G^s_{\mu}\circ H^s_{\mu^\prime}})} (\beta) = \left( f_{Q(H^s_{{\mu}^\prime})}(\beta-1)\right)^n.f_{Q(G^s_{\mu})}\left(\beta-m-\\ {\chi}_{Q({H^s_{{\mu}^\prime}})}(\beta)\right)$$
\end{corollary}
\begin{proof}
If we set $H^s_{\mu_1} \cong \ldots \cong H^s_{\mu_n} \cong {H}^s_{\mu^\prime}$ and subsequently utilizing theorem \ref{signlesslaplacianpolynomial} along with remark \ref{remark 3.3} we obtain that
\[
f_{Q({G^s_{\mu}\circ H^s_{\mu^\prime}})}(\alpha)=\left(f_{Q(H^s_{{\mu}^\prime})}(\beta-1)\right)^n.f_{Q(G^s_{\mu})}\left(\beta-m-\\ {\chi}_{Q({H^s_{{\mu}^\prime}})}(\beta)\right) \qedhere
\]
\end{proof}
\begin{corollary}  {\label{corollary 6.1.2}}
     Let $G^s_{\mu}$ be a signed graph of order $n$ and $H_1,\dots,  H_n$ be $n$ signed graphs, each of them has order $m$ such that ${\chi}_{Q({H^s_{\mu_1}})}(\beta)={\chi}_{Q({H^s_{\mu_2}})}(\beta)\dots={\chi}_{Q({H^s_{\mu_n}})}(\beta)={\chi}_{Q({H^s_{{\mu}^\prime}})}(\beta)$. Then
       $$f_{Q({G^s_\mu\circ \overset{n}{\underset{l=1}{ \Lambda}}H^s_{\mu_l}})}(\beta)=\left(\prod_{l=1}^{n} f_{Q({H^s_{\mu_l}})}(\beta-1)\right).f_{Q(G^s_{\mu})}\left(\beta-m-\\ {\chi}_{Q({H^s_{{\mu}^\prime}})}(\beta)\right).$$
     \end{corollary}
     \begin{proof}
         It is clear from Theorem \ref{signlesslaplacianpolynomial} and Remark \ref{remark 3.3}.
     \end{proof}

\begin{corollary}
Let $G^s_\mu$ be a signed graph of order $n$ and $H^s_{\mu_1}, H^s_{\mu_2},\dots, H^s_{\mu_{n}}$ be $n$ co-regular signed graphs each with order m and common co-regularity pair $(r, k)$. Then
$$f_{Q({G^s_\mu \circ \overset{n}{\underset{l=1}{ \Lambda}}H^s_{\mu_l}})}(\beta)= (\prod_{l=1}^{n} f_{Q({H^s_{\mu_l}})}(\beta-1)).f_{Q(G^s_\mu)}\left(\beta-m-\frac{m}{\beta-1-2d^{+}}\right)$$
\end{corollary}
where $d^{+}$ is the positive degree of any vertex of $H^s_{{\mu}_{l}}$ $(1 \leq l \leq n)$.
\begin{proof}
It is straightforward from proposition \ref{proposition 3.5} and Corollary \ref{corollary 6.1.2}.
\end{proof}
 The implications of Theorem \ref{signlesslaplacianpolynomial} directly validate the assertions made in the subsequent Corollaries.
\begin{corollary}
Let $G^s_{{\mu ^\prime}}$ and $G^s_{\mu^{\prime\prime}}$ be two $Q$-cospectral signed graphs, each of them has order $n$ and $H^s_{\mu_1}, H^s_{\mu_2},\dots, H^s_{\mu_n}$ be $n$ signed graphs each of them has order $m$. Then $G^s_{\mu^\prime}\circ \overset{n}{\underset{l=1}{ \Lambda}}H^s_{\mu_l}$  and $G^s_{\mu^{\prime\prime}}\circ \overset{n}{\underset{l=1}{ \Lambda}}H^s_{\mu_l}$ are $Q$-cospectral. 
\end{corollary}

 \begin{corollary}
 Let $G^s_\mu$ be a signed graph of order n and $H^s_{\mu_1}, H^s_{\mu_2},\dots, H^s_{\mu_{2n}}$ be a family of $Q$-cospectral graphs each of them has order $m$. Then $G^s_\mu\circ \overset{n}{\underset{l=1}{\Lambda}}H^s_{\mu_l}$  
 and $G^s_\mu\circ \overset{2n}{\underset{l=n+1}{ \Lambda}}H^s_{\mu_l}$ are $Q$-cospectral. 
 \end{corollary}
 
 \begin{theorem}{\label{Theorem 6.2}}
     Let $G^s_\mu$ be a signed bipartite graph with $|V(G)|=n=2k$ where $k$ denotes the size of each part. Let the underlying graph of G be r-regular. Suppose $H^s_{\mu_1} \cong \ldots \cong H^s_{\mu_k} \cong Z^s_{{\mu}^1}$ and $H^s_{\mu_{k+1}} \cong \ldots \cong H^s_{\mu_n} \cong Z^s_{{\mu}^{2}}$.  If $|V(Z^s_{\mu^1})|=m$ and $|V(Z^s_{\mu^{2}})|=s$ then

     $$f_{Q({G^s_\mu \circ \overset{n}{\underset{l=1}{ \Lambda}}H^s_{\mu_l}})}(\beta)=\prod_{j=1}^{2} ( f_{Q({Z^s_{\mu^{j}}})}(\beta-1))^k.f_{G^s_\mu}\left(\sqrt{\left(\beta-m-r-{\chi}_{Q({Z^s_{\mu^1}})}(\beta)\right)\left(\beta-s-r-{\chi}_{Q({Z^s_{\mu^2}})}(\beta)\right)}\right)$$
     \end{theorem}
    \begin{proof}
        The proof is equivalent to the one presented in Theorem \ref{Theorem 5.2}.
    \end{proof} 
    \begin{corollary}  {\label{corollary 6.2.1}}
          Let $G^s_\mu$ be a signed co-regular bipartite graph of co-regularity pair $(r,k)$ and with $|V(G)|=n=2k$ where $k$ denotes the
size of each part. Suppose $H^s_{\mu_1} \cong \ldots \cong H^s_{\mu_k} \cong Z^s_{{\mu}^1}$ and $H^s_{\mu_{k+1}} \cong \ldots \cong H^s_{\mu_n} \cong Z^s_{{\mu}^{2}}$. If $|V(Z^s_{\mu^1})|=m$ and $|V(Z^s_{\mu^{2}})|=s$ then
         $$f_{Q({G^s_\mu \circ \overset{n}{\underset{l=1}{ \Lambda}}H^s_{\mu_l}})}(\beta)=\prod_{j=1}^{2} ( f_{Q({Z^s_{\mu^{j}}})}(\beta-1))^k.f_{G^s_\mu}\left(\sqrt{\left(\beta-m-r-\frac{m}{\beta-1-2{d^+_1}}\right)\left(\beta-s-r-\frac{s}{\beta-1-2{d^+_2}}\right)}\right)$$
         where ${d}^+_1$ and ${d}^+_2$ are the positive degree of any vertex of $Z^s_{\mu^1}$ and $Z^s_{\mu^2}$ respectively. 
     \end{corollary}	
\begin{proof}
       It is straightforward from theorem \ref{Theorem 6.2} and proposition \ref{proposition 3.5}. 
\end{proof}   
    \begin{corollary}{\label{corollary 3.11.2}}
         Let $G^s_\mu$ be a signed co-regular bipartite graph of co-regularity pair $(r, k)$ and with $|V(G)|=n=2k$ where $k$ denotes the
size of each part. Suppose $H^s_{\mu_1} \cong \ldots \cong H^s_{\mu_k} \cong \bar{K}_m$ and $H^s_{\mu_{k+1}} \cong \ldots \cong H^s_{\mu_n} \cong \bar{K}_s$. Then the following holds:
         $$ f_{Q({G^s_\mu \circ \overset{n}{\underset{l=1}{ \Lambda}}H^s_{\mu_l}})}(\beta) = {\left( \beta -1 \right)}^{k(m+s)}.f_{G^s_\mu} \left(\sqrt{ \left( \beta -m-r- \frac{m}{\beta-1}\right) \left( \beta -s-r- \frac{s}{\beta-1}\right)}\right)$$          
     \end{corollary}
\begin{proof}
    It is a direct consequence of corollary \ref{corollary 6.2.1} and definition \ref{definition 4.2}.
\end{proof}    
  	 \begin{corollary}
     Let $G^s_\mu$ be a signed co-regular bipartite graph of co-regularity pair $(r, k)$ and with $|V(G)|=n=2k$ where k is the size of each part. Let $H^s_{\mu_1} \cong \ldots \cong H^s_{\mu_k} \cong \bar{K}_m$ and $H^s_{\mu_{k+1}} \cong \ldots \cong H^s_{\mu_n} \cong \phi$. Then the following holds:
     $$f_{Q({G^s_\mu \circ \overset{n}{\underset{l=1}{ \Lambda}}H^s_{\mu_l}})}(\beta)={\left( \beta-1\right) }^{km}.f_{G^s_\mu }\left(\sqrt{ \left(\beta-m-r-\frac{m}{\beta-1}\right)\left(\beta-r\right)}\right).$$
     \end{corollary}
	 \begin{proof}
    It can be easily obtained by corollary \ref{corollary 3.11.2} for $s=0$.    
    \end{proof}		

\noindent

In conclusion, our research has significantly expanded the scope of the generalized corona product, initially defined for unsigned graphs, to encompass signed graphs. This extension has allowed us to delve into the structural and spectral characteristics of signed graphs. Specifically, we have formulated statistical properties, such as edge counts and triangle categorizations, tailored for the generalized corona product of signed graphs. These statistical properties have enabled us to establish sufficient conditions for the generalised corona product of signed graphs to be unbalanced.\\\\
Furthermore, our investigation has yielded explicit expressions for the signed coronal, L-signed coronal, and Q-signed coronal, with a particular focus on specific structured signed graphs. Additionally, we have obtained the characteristic polynomial, Laplacian polynomial, and Signless Laplacian polynomial for the generalised corona product of signed graphs, all expressed in terms of their respective signed coronal, L-signed coronal and Q-signed coronal.\\\\
Moreover, we have introduced crucial criteria to identify situations where generalised corona product signed graphs may exhibit cospectral properties, including L-cospectrality and Q-cospectrality.


\end{document}